\documentclass[a4paper,11pt]{amsart}
\usepackage[utf8]{inputenc}
\usepackage{csquotes}
\usepackage{amsmath, amssymb, amsthm, xcolor,enumerate, enumitem}
\usepackage[english]{babel}
\usepackage{tikz-cd}
\usepackage[all]{xy}
\usepackage{microtype}
\usepackage[colorlinks=true, citecolor=blue, linkcolor=blue, bookmarks=true]{hyperref}
\usepackage{bbm}
\usepackage{graphicx}
\usepackage{adjustbox}
\usepackage{comment}
\usepackage{hyperref}
\usepackage{amsaddr}

\newcounter{theorem}
\newtheorem{theorem}[theorem]{Theorem}
\newtheorem{lemma}[theorem]{Lemma}
\newtheorem{prop}[theorem]{Proposition}
\newtheorem{cor}[theorem]{Corollary}

\theoremstyle{definition}
\newtheorem{defn}[theorem]{Definition}

\theoremstyle{remark}
\newtheorem*{remark*}{Remark}
\newtheorem{rmk}[theorem]{Remark}

\numberwithin{equation}{section}
\allowdisplaybreaks
\renewcommand{\labelenumi}{(\roman{enumi})}

\newcommand{\C}{\mathrm{C}^*}
\newcommand{\Z}{\mathcal{Z}}
\newcommand{\id}{\mathrm{id}}
\newcommand{\h}{\mathrm{ht}}
\newcommand{\ha}{\mathrm{hat}}
\newcommand{\Aut}{\mathrm{Aut}}
\newcommand{\End}{\mathrm{End}}

\newcommand{\Ad}{\mathrm{Ad}}
\newcommand{\dr}{\mathrm{dr}}

\newcommand{\rc}{\mathrm{rc}}
\newcommand{\rqd}{\mathrm{rqd}}
\newcommand{\hqd}{\mathrm{hqd}}
\newcommand{\s}{\mathrm{sc}}
\newcommand{\rk}{\mathrm{rank}}
\newcommand{\nuc}{\mathrm{nuc}}
\newcommand{\tp}{\mathrm{top}}
\newcommand{\rcp}{\mathrm{rcp}}
\newcommand{\Aff}{\mathrm{Aff}}
\newcommand{\diam}{\mathrm{diam}}
\newcommand{\home}{\mathrm{Homeo}}
\newcommand{\dimnuc}{\dim_{\mathrm{nuc}}}
\newcommand{\eps}{\varepsilon}

\newcommand{\cF}{\mathcal{F}}
\newcommand{\cE}{\mathcal{E}}
\newcommand{\cG}{\mathcal{G}}
\newcommand{\cH}{\mathcal{H}}
\newcommand{\cU}{\mathcal{U}}
\newcommand{\cV}{\mathcal{V}}
\newcommand{\cP}{\mathcal{P}}
\newcommand{\cB}{\mathcal{B}}
\newcommand{\cO}{\mathcal{O}}
\newcommand{\N}{\mathbb{N}}
\newcommand{\zz}{\mathbb{Z}}

\newcommand{\rr}{\mathbb{R}}
\newcommand{\cc}{\mathbb{C}}

\DeclareMathOperator{\Span}{\mathrm{span}}

\newcommand{\nn}[1]{\textcolor{red}{[[#1]]}}

\newcommand{\bj}[1]{\textcolor{teal}{[[#1]]}}

\newcommand{\rn}[1]{\textcolor{olive}{[[#1]]}}

\title{Brown--Voiculescu entropy revisited}

\author[B.\ Jacelon]{Bhishan Jacelon}
\address{Institute of Mathematics of the Czech Academy of Sciences, Žitná 25, 115 67 Prague 1, Czech Republic}
\email{jacelon@math.cas.cz}

\author[R.\ Neagu]{Robert Neagu}
\address{Department of mathematics, KU Leuven, Celestijnenlaan 200B, 3001, Leuven, Belgium}
\email{robert.neagu@kuleuven.be}

\thanks{RN funded by the European Union. Views and opinions expressed are however those of the authors only and do not necessarily reflect those of the European Union or the European Research Council. Neither the EU nor the ERC can be held responsible for them.}

\begin{document}

\begin{abstract}
 Aided by the tools and outlook provided by modern classification theory, we take a new look at the Brown--Voiculescu entropy of endomorphisms of nuclear $\C$-algebras. In particular, we introduce `coloured' versions of noncommutative topological entropy suitable for $\C$-algebras $A$ of finite nuclear dimension or finite decomposition rank. In the latter case, assuming further that $A$ is simple, separable, unital, satisfies the UCT and has finitely many extremal traces, we prove a variational type principle in terms of quasidiagonal approximations relative to this finite set of traces. Building on work of Kerr, we also show that infinite entropy occurs generically among endomorphisms and automorphisms of certain classifiable $\C$-algebras that function as noncommutative spaces of observables of topological manifolds.
\end{abstract}

\maketitle

\numberwithin{theorem}{section}	

\section*{Introduction}
\renewcommand*{\thetheorem}{\Alph{theorem}}

In its various guises, whether measure theoretic \cite{Kolmogorov58,Sinai59} or topological \cite{AdlerKonheimMcAndrew,Bowen71}, entropy is a numerical invariant that can be used to distinguish isomorphism classes of dynamical systems. In certain cases, it is even a classifying invariant (see, for example, \cite{Ornstein70}). Topological entropy, which is the focus of this article, measures the exponential growth rate of distinguishable orbit segments $(x,Tx,\dots,T^{n-1}x)$ of a topological dynamical system $(X,T)$. This point of view, in contrast to earlier operator algebraic treatments \cite{ConnesStormer,ConnesNarnhoferThirring}, motivated Voiculescu's definition \cite{Voiculescu95} of noncommutative topological entropy for automorphisms of nuclear $\C$-algebras and subsequently Brown's extension \cite{Brown99} to endomorphisms of exact $\C$-algebras.

The $\C$-algebras of primary interest to us are those that are simple, separable, unital, nuclear, tensorially absorb the Jiang--Su algebra $\Z$ and satisfy the universal coefficient theorem (UCT). These $\C$-algebras are sometimes called `classifiable' as they are now known to be classified up to isomorphism by an invariant consisting of $K$-theoretic and tracial data (see \cite{CGSTW23}, but also note that historical progress in the Elliott programme was incremental, indeed the work of many mathematicians over many years). It is in this setting that we aim to investigate the behaviour of noncommutative entropy of endomorphisms.


An important tool in the development of the classification theorem was noncommutative dimension theory, specifically the decomposition rank $\dr$ and nuclear dimension $\dimnuc$ developed respectively in \cite{KirchbergWinterDecompRank} and \cite{WinterZachNucDim}. Based on (and extending) the Lebesgue covering dimension of topological spaces, these notions of dimension quantify the ability of nuclear $\C$-algebras to admit completely positive (cp) approximation systems that can be decomposed into finitely many colours corresponding to orthogonality-preserving cp maps (see Definition~\ref{defn: NucDim}). Finite nuclear dimension played a crucial role in the first complete proof of the classification theorem (see \cite[Corollary D]{TikuisisWhiteWinter17} and \cite[Theorem 4.10]{EGLN16}), even though the later presentation \cite{CGSTW23} instead prioritises $\Z$-stability. Of course, it is now well known that these two properties are equivalent for nonelementary, simple, separable, nuclear $\C$-algebras (see \cite{Winter12,Tikuisis12,CPOU21,CastillejosEvington20}). Moreover, finite nuclear dimension or decomposition rank is sometimes easier to verify in practice, for example for $\C$-algebras arising as inductive limits of subhomogeneous building blocks.

Since the growth rates captured by the Brown--Voiculescu entropy $\h$ are exactly measured by cp approximations (see Definition~\ref{defn: BrownVoiculescuEntropy}), it seems natural to wonder whether this decomposability can be incorporated into the measurement, consequently giving rise to coloured versions of noncommutative topological entropy. This is exactly what we do in Definition~\ref{defn: NucDimAndDrEntropy}, which introduces our coloured entropy $\h_\nuc$ and contractive coloured entropy $\h_\dr$. 
These quantities dominate $\h$ (see Proposition~\ref{prop: IneqBrownVoicEntr}), and in many cases coincide with it. (Actually, we have no examples where $\h$, $\h_\nuc$ and $\h_\dr$ are different finite numbers; see Section~\ref{sect: Questions}.) In particular, in the setting of finite-dimensional compact metrisable spaces, all three notions $\h$, $\h_\nuc$ and $\h_\dr$ agree with the usual topological entropy $h_\tp$ (Theorem~\ref{thm: CommCase}), thus providing a new description of $h_\tp$ in terms of coloured refinements of open covers (see Definition~\ref{defn: ColEntr}).

For approximately finite-dimensional (AF) algebras, $\h_\nuc$ and $\h_\dr$ coincide with Voiculescu's topological approximation entropy $\ha$ (see Proposition~\ref{prop: AFEntropy}), which is a reflection of the fact that AF-algebras are precisely the separable $\C$-algebras of nuclear dimension (or decomposition rank) zero. This allows us to verify that $\h_\nuc(\alpha_k)=\h_\dr(\alpha_k)=\log k$ for the noncommutative Bernoulli shift $\alpha_k$ acting on the uniformly hyperfinite (UHF) algebra $\mathbb{M}_k(\mathbb{C})^{\otimes\zz}$ (see Proposition~\ref{prop: Shift}).

One of the well-known features of classical entropy of a topological dynamical system $(X,T)$ is the variational principle relating topological entropy to measure-theoretic entropy via $h_\tp(T)=\sup_{\mu}h_\mu(T)$, the supremum being taken over $T$-invariant Borel probability measures $\mu$. At the same time, typified by the groundbreaking work of Matui and Sato \cite{MatuiSato12,MatuiSato14} (see also \cite{Schafhauser20,CETW21,CGSTW23}), one of the hallmarks of modern classification theory is the transfer of structure from von Neumann algebras to their $\C$-counterparts, that is, from noncommutative measure theory to noncommutative topology. With this context in mind, in Section~\ref{sect: VarPr} we provide a variational type principle of a different flavour to previously established noncommutative results like \cite{NeshveyevStromer00} that involve the Connes--Narnhofer--Thirring entropy relative to invariant states. Instead, we offer an interpretation of the interplay between finite decomposition rank and tracial quasidiagonality (see Definition~\ref{defn: AmenQDTraces}). Specifically, a careful inspection of the proof of \cite[Theorem 7.5]{BBSTW19} (which, combined with \cite[Theorem A]{TikuisisWhiteWinter17}, shows that $\dr(A)\le1$ for a simple, separable, unital, nuclear, $\Z$-stable $\C$-algebra $A$ whose trace space $T(A)$ is a Bauer simplex) allows us to relate our contractive coloured entropy $\h_\dr$ to a `quasidiagonal entropy' $\hqd_\tau$ relative to traces $\tau\in T(A)$ (see Definition~\ref{defn: TraceEntropy}), assuming that the extreme tracial boundary $\partial_e T(A)$ is finite.

\begin{theorem} \label{thm: MainA}
Let $A$ be a simple, separable, unital, nuclear, $\mathcal{Z}$-stable $\C$-algebra with finitely many extremal traces and which satisfies the UCT. Suppose further that $A$ is not an AF-algebra. If $\alpha$ is an endomorphism of $A$, then \[\h_\dr(\alpha)=\sup_{\tau\in\partial_e T(A)}\hqd_\tau(\alpha)=\sup_{\tau\in T(A)}\hqd_\tau(\alpha).\]
\end{theorem}

The proof of Theorem~\ref{thm: MainA} (as well as some of the properties appearing in Sections~\ref{sect: NoncommutativeEntropy}~and~\ref{sect: Permanence}) amounts to keeping track of the finite-dimensional $\C$-algebras that appear in the relevant interlocking cp approximations. Another proof strategy that, when available, allows us to recover $\h_\nuc$ and $\h_\dr$ from $\h$ is to observe occasions where known cp approximation systems are already decomposable. See Proposition~\ref{prop: CuntzEndo}, which computes the coloured entropy of the canonical endomorphisms of Cuntz algebras. For the analysis undertaken in Section~\ref{sect: TypicalEntropy}, however, we need additional tools.

Another of the well-known aspects of classical entropy is that zero or infinite entropy tends to be typical (see, for example, \cite{GlasnerWeiss01,Yano80}). Following \cite{KerrLi05,Kerr07}, we wish to explore noncommutative instances of this phenomenon. We begin by pointing out that, as in \cite{KerrLi05}, automorphisms with zero (contractive) coloured entropy are indeed dense in $\Aut(A)$ if $A$ is a UHF-algebra or the Cuntz algebra $\cO_2$ (see Corollary~\ref{cor: DenseZeroEntropy}), the key properties of these classifiable $\C$-algebras being that they have real rank zero and no nontrivial automorphisms of their classifying invariants. Next, we recall in Theorem~\ref{thm: DenseInfEntropy} that infinite entropy (in any of the three senses) is a dense property in $\Aut(A)$ whenever $A$ is unital and $\Z$-stable (an observation made in \cite{Kerr07} and essentially contained in \cite{KerrLi05}). Finally we ask, when is infinite entropy a \emph{generic} property (that is, holds on a dense $G_\delta$ set of endomorphisms or automorphisms)?

To address this, we employ the framework of \cite{Jacelon22}, in which stably finite, classifiable $\C$-algebras $A$ are viewed as noncommutative spaces of observables of their trace spaces $T(A)$. This is a particularly fruitful viewpoint when $T(A)$ is a Bauer simplex (that is, has compact boundary $\partial_e T(A)$) and the ordered abelian group $(K_0(A),K_0(A)_+,[1_A])$ admits a unique state. This latter property can be thought of as a connectedness condition, as it holds for a commutative $\C$-algebra $C(X)$ precisely when $X$ is connected (a fact that informs the construction of the models presented in \cite[Theorem 4.4]{Jacelon22}). With these additional assumptions, \cite[Theorem 3.4]{Jacelon22} explains how to use the classification of morphisms \cite{CGSTW23} to lift generic properties of tracial dynamics to the level of the observing $\C$-algebra. Combined with Kerr's method of producing lower bounds for Brown--Voiculescu entropy (see \cite[Remark 3.10]{Kerr04}, \cite[Section 2]{Kerr07} and also Lemma~\ref{lemma: KerrLowerBound} below), and Yano's method \cite{Yano80} of exhibiting generic infinite entropy among (invertible) continuous self-maps of (nontrivial) compact topological manifolds, we arrive at the following partial answer to our question.

\begin{theorem} \label{thm: MainB}
Let $A$ be a simple, separable, unital $\C$-algebra that has finite nuclear dimension and satisfies the UCT. Suppose also that the extreme boundary of $T(A)$ is homeomorphic to a compact topological manifold (with or without boundary), and that $(K_0(A),K_0(A)_+,[1_A])$ admits a unique state. Then, $\h_\dr=\h_\nuc=\h=\infty$ on a dense $G_\delta$ subset of the set of unital endomorphisms $\alpha$ of $A$ for which $T(\alpha)$ preserves $\partial_e T(A)$. If in addition the dimension of $\partial_e T(A)$ is at least two, then $\h_\dr=\h_\nuc=\h=\infty$ on a dense $G_\delta$ subset of $\Aut(A)$.
\end{theorem}

We remark in passing that, appealing to \cite{GlasnerWeiss01} instead of \cite{Yano80}, we would get a corresponding version of Theorem~\ref{thm: MainB} assuming that the extreme tracial boundary $\partial_e T(A)$ is homeomorphic to the Hilbert cube $[0,1]^\infty$. Indeed, the proof of \cite[Theorem 1.5]{GlasnerWeiss01}, which shows that infinite entropy is generic among homeomorphisms of $[0,1]^\infty$, proceeds via local perturbation around fixed points to create topological horseshoes. Our proof is therefore valid in this situation as well (see Theorem~\ref{thm: GenericInfEntropy}).

\subsection*{Organisation}
Section~\ref{sect: Prelim} recalls the definition of the nuclear dimension and decomposition rank of $\C$-algebras and quasidiagonality of traces. In Section~\ref{sect: NoncommutativeEntropy}, we define (contractive) coloured entropy, compare it to the entropy of Voiculescu and Brown, and examine noncommutative Bernoulli shifts and canonical endomorphisms of Cuntz algebras. Section~\ref{sect: Permanence} is devoted to basic permanence properties including monotonicity (at least for restriction to invariant full hereditary subalgebras, these being guaranteed to inherit the nuclear dimension or decomposition rank of the enveloping $\C$-algebra; see Proposition~\ref{prop: Hereditary}), a Kolmogorov--Sinai type result (Proposition~\ref{prop: KolmogorovSinai}) and conjugacy invariance (Proposition~\ref{prop: Conjugacy}). In Section~\ref{sect: Commutative}, we show that (contractive) coloured entropy coincides with topological entropy in the commutative setting. Section~\ref{sect: VarPr} is dedicated to quasidiagonal entropy and Theorem~\ref{thm: MainA}. Section~\ref{sect: TypicalEntropy} explores typical values of entropy and contains Theorem~\ref{thm: MainB}. We conclude the article by compiling a list of open questions in Section~\ref{sect: Questions}.

\subsection*{Acknowledgements}
BJ was supported by the Czech Science Foundation (GA\v{C}R) project 25-15403K
and the Institute of Mathematics of the Czech Academy of Sciences (RVO: 67985840). RN was supported by the European Research Council under the European Union's Horizon Europe research and innovation programme (ERC grant AMEN-101124789) and by the postdoctoral fellowship 1204626N of the Research Foundation Flanders (FWO).
Part of this work was completed during a visit of RN to BJ at the Institute of Mathematics of the Czech Academy of Sciences.
RN would like to thank BJ and the Institute in Prague for the hospitality during the stay.  

For the purpose of open access, the authors have applied a CC BY public copyright license to any author accepted manuscript version arising from this submission.
\allowdisplaybreaks

\section{Preliminaries} \label{sect: Prelim}
\numberwithin{theorem}{section}





\subsection{Notation}

\begin{itemize}
\item We will write $\zz_+$ for the set of nonnegative integers, and $\N$ for the nonzero elements of $\zz_+$.
    \item $\mathbb{M}_n(\mathbb{C})$ will denote the $\C$-algebra of $n\times n$-matrices with coefficients in $\mathbb{C}$, and $\mathcal{K}$ will stand for the $\C$-algebra of compact operators on a separable, infinite-dimensional Hilbert space.
    \item We will write cp and cpc for completely positive and completely positive contractive, respectively. Moreover, ucp will stand for unital and completely positive.
    \item Given a $\C$-algebra $A$, we will write $\End(A)$ for the set of endomorphisms of $A$ and $\Aut(A)$ for the automorphism group of $A$, that is, $\End(A)=\{^*\text{-homomorphisms}\ A\to A\}$ and $\Aut(A)=\{^*\text{-isomorphisms}\ A\to A\}$.
    \item Given a $\C$-algebra $A$ and a free ultrafilter $\omega$ on $\mathbb{N}$, we will write $A_\omega$ for the $\C$-algebra obtained as a quotient of the bounded sequences of $A$ by the null sequences along the ultrafilter $\omega$.
    \item By a `trace' we will always mean a tracial state. For a unital $\C$-algebra $A$, we will denote the set of traces on $A$ by $T(A)$, and the set of extremal traces by $\partial_e T(A)$, both sets equipped with the weak$^*$-topology.
\end{itemize}

\subsection{Nuclear dimension and decomposition rank}

Before defining the notion of finite nuclear dimension for $\C$-algebras, we need to recall the definition of order zero maps. A cp map $\phi: A\to B$ between $\C$-algebras is said to be
\emph{order zero} if, for every $a,b \in A_+$ with $ab=0$, one has $\phi(a)\phi(b)=0$. The general theory for these maps was developed in \cite{WinterZachOrderZero}. 

\begin{defn}[{\cite{KirchbergWinterDecompRank,WinterZachNucDim}}]\label{defn: NucDim}
Let $A$ be a $\C$-algebra and $d$ be a nonnegative integer.
It is said that $A$ has \emph{nuclear dimension at most $d$}, and denoted by $\dimnuc(A)\leq d$, if for any finite set $\cF\subseteq A$ and any $\eps>0$, there exists a finite-dimensional $\C$-algebra $F=F^{(0)}\oplus \ldots \oplus F^{(d)}$  and  maps $\psi\colon A\to F$ and $\varphi\colon F\to A$ such that 

\begin{enumerate}
    \item $\|\varphi\circ\psi(a)-a\|\leq \eps$ for all $a\in \cF$,
    \item $\psi$ is cpc,
    \item if we denote by $\varphi^{(i)}$ the restriction of $\varphi$ to $F^{(i)}$, then $\varphi^{(i)}$ is a cpc order zero map for all $i=0,\ldots, d$.
\end{enumerate}
In this case, we call $(F,\psi,\varphi)$ a \emph{$(d,\cF,\eps)$-decomposable system} for $A$.
If additionally $\varphi$ is contractive, then we call $(F,\psi,\varphi)$ a \emph{contractive $(d,\cF,\eps)$-decomposable system} for $A$. It is said that $A$ has \emph{decomposition rank at most $d$} (written $\dr(A)\leq d$) if, for every $\cF$ and $\eps$, there exists a contractive $(d,\cF,\eps)$-decomposable system for $A$.
The $\C$-algebra $A$ has nuclear dimension (respectively, decomposition rank) equal to $d$ if $d$ is the minimum number for which $\dimnuc(A)\leq d$ (respectively, $\dr(A)\leq d$). If there is no such $d$, then by definition $\dimnuc(A)=\infty$ (respectively, $\dr(A)=\infty$).
\end{defn}

Furthermore, these approximations properties for $\C$-algebras are often related to external approximation properties for traces, such as amenability and quasidiagonality. We will only define the notion of quasidiagonality for a trace. Note that the definition we choose (\cite[Definition 7.2]{BBSTW19}) is formally different than Brown's in \cite{Brown06}, but in fact equivalent as observed in the discussion after \cite[Definition 7.2]{BBSTW19}

\begin{defn}[{\cite[Definition 7.2]{BBSTW19}}]\label{defn: AmenQDTraces}
Let $A$ be a $\C$-algebra.
A trace $\tau\colon A\to\mathbb{C}$ is \emph{quasidiagonal} if for any finite set $\cF\subseteq A$ and any $\eps>0$, there is a finite-dimensional $\C$-algebra $F$ with a trace $\tau_F$ and a cpc map $\varphi\colon A\rightarrow F$ such that $$\|\varphi(ab)-\varphi(a)\varphi(b)\|\leq\eps$$ and $|\tau_F(\varphi(a))-\tau(a)|\leq\eps$ for any $a,b\in \cF$.   
\end{defn}

\section{Noncommutative topological entropy} \label{sect: NoncommutativeEntropy}

Let us first recall the definition of the Brown--Voiculescu entropy, as introduced in \cite{Voiculescu95} for unital nuclear $\C$-algebras, and refined in \cite{Brown99} to cover exact $\C$-algebras.

\begin{defn}[{\cite[Definition 4.1]{Voiculescu95}}]\label{defn: BrownVoiculescuEntropy}
Let $A$ be a nuclear $\C$-algebra and let $\alpha\colon A\to A$ be an endomorphism of $A$.
For any finite set $\cF\subseteq A$ and any $\eps>0$, we say that $(F,\psi,\varphi)$ is an \emph{$(\cF,\eps)$-completely positive system} for $A$ if $F$ is finite dimensional, $\psi\colon A\to F$ and $\varphi\colon F\to A$ are cpc, and $\|\varphi(\psi(a))-a\|\leq\eps$ for any $a\in \cF$. We then consider the \emph{$(\cF,\eps)$-completely positive rank} to be \[\rcp(\cF,\eps)=\inf\{\rk(F)\mid (F,\psi,\varphi) \ \text{is $(\cF,\eps)$-completely positive}\},\] where the \emph{rank} of a finite-dimensional $\C$-algebra is defined to be the dimension of a maximal abelian subalgebra.
We now further define the quantities
\begin{align*}
&\h(\alpha,\cF,\eps)=\limsup\limits_{n\to\infty}\frac{1}{n}\log \rcp(\cF\cup\alpha(\cF)\cup\ldots\cup\alpha^{n-1}(\cF),\eps),\\
&\h(\alpha,\cF)=\sup\limits_{\eps>0}\h(\alpha,\cF,\eps), \\
&\h(\alpha)=\sup\limits_{\cF}\h(\alpha,\cF),\ \text{with $\cF$ ranging over all finite subsets of $A$.}
\end{align*} 
The definition for exact $\C$-algebras is the same, except that the completely positive systems are required to approximate a faithful representation $A\to\mathcal{B}(H)$ rather than the identity map $A\to A$. By convention, $\h(\alpha)=\infty$ if $A$ is not exact.
\end{defn}

If the $\C$-algebras in question have finite nuclear dimension (or finite decomposition rank), we consider the following notions of coloured entropy.

\begin{defn}\label{defn: NucDimAndDrEntropy}
Let $A$ be a $\C$-algebra with nuclear dimension equal to $d$ and let $\alpha\colon A\to A$ be an endomorphism of $A$.
For any finite set $\cF\subseteq A$ and any $\eps>0$, we define the \emph{$(d,\cF,\eps)$-decomposable rank} by \[r_{\nuc}(d,\cF,\eps)=\inf\{\rk(F)\mid (F,\psi,\varphi)\ \text{is $(d,\cF,\eps)$-decomposable}\}.\] 
Moreover, we define the \emph{contractive $(d,\cF,\eps)$-decomposable rank} by \[r_{\dr}(d,\cF,\eps)=\inf\{\rk(F)\mid (F,\psi,\varphi)\ \text{is contractive $(d,\cF,\eps)$-decomposable}\}.\]
We now further define the quantities
\begin{align*}
&\h_\nuc(\alpha,\cF,\eps)=\limsup\limits_{n\to\infty}\frac{1}{n}\log r_\nuc(d,\cF\cup\alpha(\cF)\cup\ldots\cup\alpha^{n-1}(\cF),\eps),\\
&\h_\nuc(\alpha,\cF)=\sup\limits_{\eps>0}\h_\nuc(\alpha,\cF,\eps), \\
&\h_\nuc(\alpha)=\sup\limits_{\cF}\h_\nuc(\alpha,\cF),\ \text{with $\cF$ ranging over all finite subsets of $A$.}
\end{align*}We will often call the last quantity the \emph{coloured entropy}. If $A$ does not have finite nuclear dimension, then $\h_\nuc(\alpha)=\infty$ by convention. 
Using the contractive decomposable rank instead, we can likewise define $\h_\dr(\alpha,\cF,\eps)$, $\h_\dr(\alpha,\cF)$, and $\h_\dr(\alpha)$. Similarly, we will often refer to the last quantity as the \emph{contractive coloured entropy}, and $\h_\dr(\alpha)=\infty$ if $A$ does not have finite decomposition rank.
\end{defn}

\begin{prop}\label{prop: IneqBrownVoicEntr}
Let $A$ be a unital, exact $\C$-algebra and let $\alpha\colon A\to A$ be an endomorphism of $A$. Then $\h(\alpha)\leq \h_\nuc(\alpha)\leq \h_\dr(\alpha)$.  
\end{prop}

\begin{proof}
The fact that $\h_\nuc(\alpha)\leq \h_\dr(\alpha)$ follows from the definition, so it remains to show that $\h(\alpha)\leq \h_\nuc(\alpha)$. 
If $\dim_{\nuc}(A)=\infty$, then $\h_\nuc(\alpha)=\infty$, so the inequality holds. 
Therefore, we can assume that $\dim_\nuc(A)=d$ for some $d\in \mathbb{Z}_+$.

We will now follow the standard proof that a $\C$-algebra with finite nuclear dimension is nuclear.
Let $\cF$ be a finite subset of $A$, $n\in\mathbb{N}$, $0<\eps<1$, and $\delta>0$ such that if $|1-t|\leq\delta$, then $|1-t^{-1/2}|<\frac{\eps}{4(d+1)}.$ We can assume that all elements of $\cF$ are contractions and $1\in\cF$.
Let $(F,\psi,\varphi)$ be a $(d,\cF\cup\alpha(\cF)\cup\ldots\cup\alpha^{n-1}(\cF),\min\{\delta, \frac{\eps}{4},\frac{1}{2}\})$-decomposable system for $A$. In particular, this implies that $\|\varphi(\psi(1))-1\|\leq\delta$, so $\varphi(\psi(1))$ is invertible if $\delta$ is small enough. Moreover, by choice of $\delta$, the continuous functional calculus implies that \[\|1-\varphi(\psi(1))^{-1/2}\|<\frac{\eps}{4(d+1)}.\]Using \cite[Lemma 2.2.5]{BrownOzawa08}, we can now choose a ucp map $\psi_0\colon A\to F$ such that \[\psi(a)=\psi(1)^{1/2}\psi_0(a)\psi(1)^{1/2}, \ a\in A.\]Then, we let $\varphi_0\colon F\to A$ given by \[\varphi_0(x)=\varphi(\psi(1))^{-1/2}\varphi\left(\psi(1)^{1/2}x\psi(1)^{1/2}\right)\varphi(\psi(1))^{-1/2}.\]As in the proof of \cite[Proposition 1.5.4]{Castillejos16}, the maps $\psi_0$ and $\varphi_0$ are ucp and their composition approximates the identity map on $\cF\cup\alpha(\cF)\cup\ldots\cup\alpha^{n-1}(\cF)$ up to $\eps$. Since this procedure did not change the finite-dimensional $\C$-algebra $F$, it follows that \[\h(\alpha,\cF,\eps)\leq\h_\nuc(\alpha,\cF,\min\{\delta, \frac{\eps}{4},\frac{1}{2}\})\leq \h_\nuc(\alpha).\]Since $\cF$ and $\eps$ were arbitrary, this implies that $\h(\alpha)\leq\h_\nuc(\alpha)$.
\end{proof}

Next, we show that in the approximately finite-dimensional setting, both the coloured and the contractive coloured entropies coincide with the \emph{topological approximation entropy} $\ha$ introduced in \cite{Voiculescu95}. The definition of this version of noncommutative entropy is similar to Definition~\ref{defn: BrownVoiculescuEntropy}, except that the finite-dimensional $\C$-algebras appearing in the infimum do not come from completely positive approximation systems. Instead, bearing in mind that the ambient $\C$-algebra $A$ is AF, we take the infimum of the ranks of finite-dimensional $\C$-subalgebras $F$ of $A$ that $\eps$-contain the given finite set $\cF$ (that is, for every $a\in\cF$ there exists $b\in F$ such that $\|a-b\|<\eps$). The rest of the definition is unchanged.

\begin{prop} \label{prop: AFEntropy}
Let $A$ be a separable, unital AF-algebra, and let $\alpha$ be an endomorphism of $A$. Then $\h_{\dr}(\alpha) = \h_{\nuc}(\alpha) = \ha(\alpha)$.
\end{prop}

\begin{proof}
Fix $\eps>0$, a finite set $\cF\subseteq A$, and $n\in\mathbb{N}$. Suppose that $F$ is a finite-dimensional subalgebra of $A$ that $\eps$-contains $\cF\cup\alpha(\cF)\cup\ldots\cup\alpha^{n-1}(\cF)$ (in the sense described above). As in the proof of \cite[Proposition 4.5]{Voiculescu95}, we then get an $(\cF\cup\alpha(\cF)\cup\ldots\cup\alpha^{n-1}(\cF),2\eps)$-completely positive system $(F,\psi,\varphi)$, where $\psi\colon A \to F$ is an Arveson extension of $\id\colon F \to F$ and $\varphi$ is the inclusion map $F \to A$. Since $\varphi$ is a $^*$-homomorphism, so in particular is order zero and contractive, it follows that $\h_\dr(\alpha)\le\ha(\alpha)$. For the reverse inequality, we recall the proof that $\dr(A)=0$ (or equally, $\dimnuc(A)=0$) implies that $A$ is an AF-algebra. The argument can be traced back to \cite[Theorem 3.4]{Winter03}, and as observed there, for unital $A$ amounts to the observation that an approximately unital cpc order zero map $\varphi\colon F \to A$ can be perturbed to a $^*$-homomorphism $\varphi'\colon F\to A$. It follows that $\ha(\alpha)\le\h_\nuc(\alpha)\le\h_\dr(\alpha)$.
\end{proof}

Using Proposition~\ref{prop: AFEntropy}, we can compute the coloured and the contractive coloured entropies for an important class of examples, the noncommutative Bernoulli shifts.

\begin{prop} \label{prop: Shift}
Let $A=B^{\otimes\zz}$ for a given UHF-algebra $B$, and let $\alpha\colon A \to A$ be the tensor shift automorphism induced by the map $i\mapsto i+1$ on the index set $\zz$. Then either $B\cong\mathbb{M}_k(\mathbb{C})$ for some $k\in\N$, in which case $\h_\dr(\alpha) = \h_\nuc(\alpha) = \h(\alpha) = \log k$, or $B$ is infinite dimensional and $\h_\dr(\alpha) = \h_\nuc(\alpha) = \h(\alpha) = \infty$.
\end{prop}

\begin{proof}
For infinite-dimensional $B$, $\alpha$ restricts to the shift on subalgebras of the form $\mathbb{M}_k(\mathbb{C})^{\otimes\zz}$ for arbitrarily large $k$. The second assertion therefore follows from the first by monotonicity of $\h$ (see \cite[Proposition 2.1]{Brown99} and also Proposition~\ref{prop: Hereditary} below) and Proposition~\ref{prop: IneqBrownVoicEntr}. So let us assume that $B=\mathbb{M}_k(\mathbb{C})$. By \cite[Propositions 2.6 and 4.7]{Voiculescu95}, $\ha(\alpha)=\h(\alpha)=\log k$. By Proposition~\ref{prop: AFEntropy}, we then have $\h_\dr(\alpha)=\h_\nuc(\alpha)=\log k$ as well.
\end{proof}

The \emph{canonical endomorphism} of the Cuntz algebra $\cO_k=\C(s_1,\dots,s_k)$, $2\le k<\infty$, is the map $\theta_k\colon \cO_k \to \cO_k$ defined by $\theta_k(a) = \sum_{i=1}^k s_ias_i^*$. It leaves invariant the canonical UHF subalgebra $\mathbb{M}_k(\mathbb{C})^{\otimes\N}\subseteq\cO_k$, on which it acts as the one-sided Bernoulli shift $a\mapsto1\otimes a$. This observation forms part of the proof of \cite[Theorem 4.6]{Choda99}, which computes $\h(\theta_k)=\log k$. Appealing to the alternative proof provided by \cite{BocaGoldstein00}, we discover that $\h_\nuc(\theta_k)=\log k$ as well.

\begin{prop} \label{prop: CuntzEndo}
If $\theta_k\colon \cO_k \to \cO_k$ is the canonical endomorphism described above, then $\h_\dr(\theta_k)=\infty$ and $\h_\nuc(\theta_k)=\log k$.
\end{prop}

\begin{proof}
The equality $\h_\dr(\theta_k)=\infty$ holds because the purely infinite $\C$-algebra $\cO_k$ has infinite decomposition rank. We turn to $\h_\nuc$. In conjunction with Choda's theorem \cite[Theorem 4.6]{Choda99}, Proposition~\ref{prop: IneqBrownVoicEntr} gives that $\h_\nuc(\theta_k)\ge \h(\theta_k)=\log k$. The reverse inequality can be extracted from the proof of  \cite[Proposition 3]{BocaGoldstein00}. Given $\eps>0$, a finite set $\cF\subseteq \cO_k$ (which can be assumed to consist of monomials in the generators $s_1,\dots,s_k$), this proposition provides a sequence $(F_n,\psi_n,\varphi_n)$ of $(\cF\cup\theta_k(\cF)\cup\ldots\cup\theta_k^{n-1}(\cF),\eps)$-completely positive systems such that $\limsup_{n\to\infty}\frac{1}{n}\log \rk(F_n) = \log k$. The starting point in the construction of these systems is an $(\cF,\eps)$-system $(F_0,\psi_0,\varphi_0)$ obtained from nuclearity of $\cO_k$. Since $\dimnuc(\cO_k)=1$, we can assume that $(F_0,\psi_0,\varphi_0)$ is in fact $(1,\cF,\eps)$-decomposable. For each $n$, $\varphi_n$ and $\psi_n$ are defined in terms of an isomorphism $\rho_m \colon \cO_k \to \mathbb{M}_{k^m}(\mathbb{C}) \otimes \cO_k$, where $m=n+n_0$ for a suitable $n_0$ depending on $\cF$. Specifically, we have $F_n=\mathbb{M}_{k^m}(\mathbb{C}) \otimes F_0$, $\psi_n=(\id\otimes\psi_0)\circ\rho_m$ and $\varphi_n=\rho_m^{-1}\circ(\id\otimes\varphi_0)$, so $\varphi_n$ is $1$-decomposable since $\varphi_0$ is. It follows that $\h_\nuc(\theta_k)\le \log k$ and hence that $\h_\nuc(\theta_k)=\log k$.
\end{proof}

\begin{rmk} \label{rmk: CuntzKrieger}
The results of \cite{BocaGoldstein00} apply more generally to the canonical `endomorphisms' $\theta_A$ (which are actually unital completely positive self-maps) of Cuntz--Krieger algebras $\cO_A$ associated to irreducible, non-permutation matrices $A$ taking values in $\{0,1\}$. In this setting, the map $\rho_m \colon \cO_A \to \mathbb{M}_{w(m)}\otimes\cO_A$ used in the proof of Proposition~\ref{prop: CuntzEndo} (wherein $w(m)=k^m$) is in general not an isomorphism. To define $\psi_n$, the map $\rho_m^{-1} \colon \rho_m(\cO_A) \to \cO_A \subseteq \mathcal{B}(H)$ must first be extended to a completely positive map $\mathbb{M}_{w(m)}\otimes\cO_A \to \mathcal{B}(H)$. Because this extension need not be order zero, it is unclear whether we can obtain a decomposable system in this way. We leave open the problem of determining whether $\h_\nuc(\theta_A)$ is equal to the logarithm of the spectral radius of the matrix $A$.

On the other hand, by the same reasoning used in the proof of Proposition~\ref{prop: CuntzEndo}, the entropy bound for polynomial endomorphisms of Cuntz algebras established in \cite[Theorem 2.2]{SkalskiZacharias08} carries over to $\h_\nuc$. More precisely, if $u$ is a unitary in the canonical copy of $M_{k^m}$ in $\cO_k$ and $\rho_u\colon \cO_k\to\cO_k$ is the associated endomorphism defined by $s_i\mapsto us_i$, $i=1,\dots,k$, then $\h_\nuc(\rho_u)\le(m-1)\log k$. Moreover, $\h_\nuc(\alpha)$ agrees with $\h(\alpha)$ for all of the concrete endomorphisms $\alpha$ of $\cO_2$ considered in \cite{SkalskiZacharias08}.
\end{rmk}

\section{Permanence properties} \label{sect: Permanence}

In this section, we will record some permanence properties for the (contractive) coloured entropy. Most of the arguments are essentially following the standard proofs for the Brown--Voiculescu entropy (see \cite[Section 2]{Brown99}). We will start with a form of monotonicity for the coloured entropies. Since neither nuclear dimension nor decomposition rank is inherited by $\C$-subalgebras, we will restrict ourselves to hereditary $\C$-subalgebras.

\begin{prop}\label{prop: Hereditary}
Let $A$ be a $\C$-algebra, $\alpha$ be an endomorphism of $A$, and $B\subseteq A$ be a full hereditary $\C$-subalgebra of $A$ which is invariant under $\alpha$. Then \[\h_\nuc(\alpha|_B)\leq \h_\nuc(\alpha) \ \text{and} \ \h_\dr(\alpha|_B)\leq \h_\dr(\alpha). \] 
\end{prop}

\begin{proof}
If $\dim_\nuc(A)=\infty$, then $\h_\nuc(\alpha)=\infty$, so there is nothing to prove. Therefore, we will assume that $A$ has finite nuclear dimension, say $d$. We will now follow the proof that $B$ has nuclear dimension at most $d$ provided in \cite[Proposition 2.5]{WinterZachNucDim}. Note that since $B$ is a full hereditary $\C$-subalgebra of $A$, we know that $\dim_\nuc(B)=\dim_\nuc(A)$ (\cite[Corollary 2.8(ii)]{WinterZachNucDim}). Thus, the estimates provided by the proof of \cite[Proposition 2.5]{WinterZachNucDim} can indeed be used to compute $\h_\nuc(\alpha|_B)$.

Let $\cF\subseteq B$ be a finite set of positive contractions, $\eps>0$, and $n\in\mathbb{N}$. As in \cite[Proposition 2.5]{WinterZachNucDim}, let $h_0,h_1$ be positive contractions in $B$ such that $h_0h_1=h_1$ and $h_1b=b$ for any $b\in\cF$. Set \[\cG=\cF\cup\{h_0,h_1\},\ \eta=\min\left\{\frac{\eps^8}{13(d+1)},\frac{1}{2^{16}}\right\},\] and let $(F,\psi,\varphi)$ be a $(d, \cG\cup\alpha(\cG)\cup\ldots\cup\alpha^{n-1}(\cG),\eta)$-decomposable system for $A$. One then defines a projection $p\in F$ as in the proof of \cite[Proposition 2.5]{WinterZachNucDim} and builds a decomposable system $(\hat{F},\hat{\psi},\hat{\varphi})$ for $B$, where $\hat{F}=pFp$. Note that the proof of \cite[Proposition 2.5]{WinterZachNucDim} shows that a $(d,\cG,\eta)$-decomposable system for $A$ yields a $(d,\cF,\eps)$-decomposable system for $B$. Then for any $1\leq k\leq n-1$, as $\alpha^k(h_1)\alpha^k(h_0)=\alpha^k(h_0h_1)=\alpha^k(h_1)$ and $\alpha^k(h_1)\alpha^k(b)=\alpha^k(h_1b)=\alpha^k(b)$ for any $b\in\cF$, a $(d,\alpha^k(\cG),\eta)$-decomposable system for $A$ gives a $(d,\alpha^k(\cF),\eps)$-decomposable system for $B$. Therefore, the system $(\hat{F},\hat{\psi},\hat{\varphi})$ is a $(d,\cF\cup\alpha(\cF)\cup\ldots\cup\alpha^{n-1}(\cF),\eps)$-decomposable system for $B$. Since $\mathrm{rank}(\hat{F})\leq\mathrm{rank}(F)$, we have shown that
\begin{equation}\label{eq:Hered1}
\h_\nuc(\alpha|_B,\cF,\eps)\leq\h_\nuc(\alpha,\cG,\eta)\leq \h_\nuc(\alpha).
\end{equation}
But $\cF$ and $\eps$ were arbitrary, so \eqref{eq:Hered1} yields that $\h_\nuc(\alpha|_B)\leq\h_\nuc(\alpha)$. Further assuming that $A$ has finite decomposition rank, using the proof of \cite[Proposition 3.8]{KirchbergWinterDecompRank} instead, one can show in the same fashion that $\h_\dr(\alpha|_B)\leq \h_\dr(\alpha)$.
\end{proof}


\begin{lemma}[cf.{\cite[Lemma 2.4]{Brown99}}]\label{lemma: TechLemBrown}
Let $A$ be a $\C$-algebra with nuclear dimension equal to $d$, $\alpha$ be an automorphism of $A$, $\cF\subseteq A$ finite, and $\eps>0$. Then \[r_\nuc(d,\alpha(\cF),\eps)= r_\nuc(d,\cF,\eps).\]Moreover, if $A$ has decomposition rank equal to $d$, then the same equality holds by replacing $r_\nuc$ with $r_\dr$.
\end{lemma}

\begin{proof}
Let $(F,\psi,\varphi)$ be a $(d,\cF,\eps)$-decomposable system for $A$. Considering the system $(F,\psi\circ\alpha^{-1},\alpha\circ\varphi)$, we get a $(d,\alpha(\cF),\eps)$-decomposable system for $A$, so \[r_\nuc(d,\alpha(\cF),\eps)\leq r_\nuc(d,\cF,\eps).\]Conversely, if we start with a $(d,\alpha(\cF),\eps)$-decomposable system for $A$, say $(F,\psi,\varphi)$, then $(F, \psi\circ\alpha,\alpha^{-1}\circ\varphi)$ is a $(d,\cF,\eps)$-decomposable system for $A$. This shows that \[r_\nuc(d,\alpha(\cF),\eps)\geq r_\nuc(d,\cF,\eps),\]so the conclusion follows. Further assuming that $\dr(A)=d$, exactly the same argument shows that $r_\dr(d,\alpha(\cF),\eps)=r_\dr(d,\cF,\eps)$.
\end{proof}

Using Lemma~\ref{lemma: TechLemBrown}, we can extract the behaviour of the (contractive) coloured entropy under taking powers.

\begin{prop}\label{prop: EntropyPower}
Let $A$ be a $\C$-algebra and $\alpha$ be an automorphism of $A$. Then \[\h_\nuc(\alpha^k)=|k|\h_\nuc(\alpha) \ \text{and} \ \h_\dr(\alpha^k)=|k|\h_\dr(\alpha), \ k\in\mathbb{Z}.\]
\end{prop}

\begin{proof}
If $\dim_\nuc(A)=\infty$, then $\h_\nuc(\alpha)=\infty$, so there is nothing to prove. Therefore, we will assume that $A$ has finite nuclear dimension, say $d$. We will first show that $\h_\nuc(\alpha)=\h_\nuc(\alpha^{-1})$. Let $\cF\subseteq A$ be finite, $\eps>0$, and $n\in\mathbb{N}$. Then, by Lemma \ref{lemma: TechLemBrown} we get that
\begin{align*}
r_\nuc\left(d,\bigcup_{j=0}^{n-1}\alpha^j(\cF),\eps\right)&=r_\nuc\left(d, \alpha^{-n+1}\left(\bigcup_{j=0}^{n-1}\alpha^j(\cF)\right),\eps\right)\\&=r_\nuc\left(d,\bigcup_{j=0}^{n-1}\alpha^{-j}(\cF),\eps\right),
\end{align*}
which yields that $\h_\nuc(\alpha)=\h_\nuc(\alpha^{-1})$. Therefore, it suffices to check that $\h_\nuc(\alpha^k)=k \h_\nuc(\alpha)$ for $k\geq 0$. Since the equality holds trivially for $k=0$, we shall assume $k\geq 1$. Then the same argument as in \cite[Proposition 1.3]{Voiculescu95} gives the result. Finally, exactly the same proof shows that $\h_\dr(\alpha^k)=|k|\h_\dr(\alpha)$ for any $k\in\mathbb{Z}$.
\end{proof}

Another consequence of Lemma~\ref{lemma: TechLemBrown} is a Kolmogorov--Sinai description of the (contractive) coloured entropy in terms of a generating system of finite subsets (cf.\ \cite[Proposition 3.4]{Voiculescu95} and \cite[Proposition 2.6]{Brown99}).

\begin{prop} \label{prop: KolmogorovSinai}
Let $A$ be a $\C$-algebra and $\alpha$ be an automorphism of $A$. Suppose that there is a net $(\cF_\lambda)_{\lambda\in\Lambda}$ of finite subsets of $A$ (partially ordered by inclusion) such that $\bigcup_{\lambda\in\Lambda,n\in\zz}\alpha^n(\cF_\lambda)$ is dense in $A$. Then
\[
\h_\nuc(\alpha) =\sup_{\lambda\in\Lambda}\h_\nuc(\alpha,\cF_\lambda) \ \text{and} \ \h_\dr(\alpha) =\sup_{\lambda\in\Lambda}\h_\dr(\alpha,\cF_\lambda).
\]
\end{prop}

\begin{proof}
The proof is the same as that of \cite[Proposition 2.6]{Brown99}, just using Lemma~\ref{lemma: TechLemBrown} instead of \cite[Proposition 2.4]{Brown99}.
\end{proof}

Similarly to the Brown--Voiculescu entropy, the notions of entropy introduced in Definition \ref{defn: NucDimAndDrEntropy} are well behaved with respect to taking direct sums and conjugacies. However, since in general the nuclear dimension of an inductive limit can only be bounded above by the dimensions of the building blocks, we cannot obtain general bounds for the coloured entropy of an inductive limit automorphism (but see Remark~\ref{rmk: DiffDefEntropy}). For similar reasons, we cannot provide bounds for the coloured entropy of tensor products. However, we can show the coloured entropy is stable under tensoring with the compact operators.

\begin{prop}\label{prop: TensorProd}
Let $A$ be a $\C$-algebra and $\alpha$ be an endomorphism of $A$. Then\[\h_\nuc(\alpha\otimes\id_\mathcal{K})=\h_\nuc(\alpha) \ \text{and} \ \h_\dr(\alpha\otimes\id_\mathcal{K})=\h_\dr(\alpha).\]
\end{prop}

\begin{proof}
If $\dim_\nuc(A)=\infty$, then $\h_\nuc(\alpha)=\h_\nuc(\alpha\otimes\id_\mathcal{K})=\infty$, so there is nothing to prove. Therefore, we will assume that $A$ has finite nuclear dimension, say $d$. First note that by \cite[Corollary 2.8(i)]{WinterZachNucDim}, $\dim_\nuc(A\otimes\mathcal{K})=\dim_\nuc(A)=d$. Then, Proposition \ref{prop: Hereditary} shows that \[\h_\nuc(\alpha)\leq\h_\nuc(\alpha\otimes\id_\mathcal{K}).\]

For the reverse inequality, let $\cF\subseteq A$ and $\cG\subseteq \mathcal{K}$ be finite collections of contractions, $\eps>0$, and $n\in\mathbb{N}$. Then let $(F,\psi,\varphi)$ be a $(d,\cF\cup\alpha(\cF)\cup\ldots\cup\alpha^{n-1}(\cF),\eps/2)$-decomposable system for $A$ and $(G,\gamma,\eta)$ be a $(0,\cG,\eps/2)$-decomposable system for $\mathcal{K}$. Using that the tensor of two order zero maps is order zero (\cite[Corollary 4.3]{WinterZachOrderZero}), it follows that $(F\otimes G,\psi\otimes\gamma,\varphi\otimes\eta)$ is a $(d,(\cF\otimes\cG)\cup(\alpha\otimes\id_\mathcal{K})(\cF\otimes\cG)\cup\ldots\cup(\alpha\otimes\id_\mathcal{K})^{n-1}(\cF\otimes\cG),\eps)$-decomposable system for $A\otimes\mathcal{K}$. Thus, we have shown that \[\h_\nuc(\alpha\otimes\id_\mathcal{K},\cF\otimes\cG,\eps)\leq \h_\nuc(\alpha,\cF,\eps/2)+\h_\nuc(\id_\mathcal{K},\cG,\eps/2)\leq \h_\nuc(\alpha),\]where the last inequality follows since $\h_\nuc(\id_\mathcal{K},\cG,\eps/2)=0$. Since finite sets of the form $\cF\otimes\cG$ span a dense set in $A\otimes\mathcal{K}$, it suffices to take the supremum over such sets (in a similar fashion to Proposition~\ref{prop: KolmogorovSinai}) to conclude that $\h_\nuc(\alpha\otimes\id_\mathcal{K})\leq\h_\nuc(\alpha)$. Clearly the same argument also shows that $\h_\dr(\alpha)=\h_\dr(\alpha\otimes\id_\mathcal{K})$. 
\end{proof}

\begin{prop}\label{prop: DirectSums}
Let $A$ be a $\C$-algebra and $\alpha_1,\alpha_2$ be endomorphisms of $A$. Then\[\h_\nuc(\alpha_1\oplus\alpha_2)=\max\{\h_\nuc(\alpha_1),\h_\nuc(\alpha_2)\}\] and \[\h_\dr(\alpha_1\oplus\alpha_2)=\max\{\h_\dr(\alpha_1),\h_\dr(\alpha_2)\}.\]
\end{prop}

\begin{proof}
If $\dim_\nuc(A)=\infty$, then $\h_\nuc(\alpha_1)=\h_\nuc(\alpha_1\oplus\alpha_2)=\infty$, so there is nothing to prove. Therefore, we will assume that $A$ has finite nuclear dimension, say $d$. Then $A\oplus A$ has nuclear dimension equal to $d$ by \cite[Proposition 2.3(i)]{WinterZachNucDim}. We will first show that \[\h_\nuc(\alpha_1)\leq\h_\nuc(\alpha_1\oplus\alpha_2).\]Let $\cF\subseteq A$ be finite, $\cF_0=\cF\oplus 0\subseteq A\oplus A$, $\eps>0$, and $n\in\mathbb{N}$. We then let $(F,\psi_0,\varphi_0)$ be a $(d,\cF_0\cup(\alpha_1\oplus\alpha_2)(\cF_0)\cup\ldots\cup(\alpha_1\oplus\alpha_2)^{n-1}(\cF_0),\eps)$-decomposable system for $A\oplus A$. To obtain a decomposable system for $A$, we let $\psi=\psi_0\circ\iota_1\colon A\to F$, where $\iota_1\colon A\to A\oplus A$ is the first-factor embedding, and $\varphi=\pi_1\circ\varphi_0$, where $\pi_1\colon A\oplus A\to A$ is the projection onto the first copy of $A$. Then, it is immediate to see that $(F,\psi,\varphi)$ is a $(d,\cF\cup\alpha_1(\cF)\cup\ldots\cup\alpha_1^{n-1}(\cF),\eps)$-decomposable system for $A$. Hence, \[\h_\nuc(\alpha_1,\cF,\eps)\leq \h_\nuc(\alpha_1\oplus\alpha_2, \cF\oplus 0,\eps)\leq\h_\nuc(\alpha_1\oplus\alpha_2).\] Since $\cF$ and $\eps$ were arbitrary, we get that $\h_\nuc(\alpha_1)\leq\h_\nuc(\alpha_1\oplus\alpha_2)$, and by symmetry, this shows that \[\h_\nuc(\alpha_1\oplus\alpha_2)\geq\max\{\h_\nuc(\alpha_1),\h_\nuc(\alpha_2)\}.\]

Conversely, let $\cF_1\oplus\cF_2\subseteq A\oplus A$ be finite, $\eps>0$, and $n\in\mathbb{N}$. Then, for each $i=1,2$, let $(F_i,\psi_i,\varphi_i)$ be a $(d,\cF_i\cup\alpha_i(\cF_i)\cup\ldots\cup\alpha_i^{n-1}(\cF_i),\eps)$-decomposable system for $A$. Since $(F_1\oplus F_2,\psi_1\oplus\psi_2,\varphi_1\oplus\varphi_2)$ is a $(d,(\cF_1\oplus\cF_2)\cup\ldots\cup(\alpha_1\oplus\alpha_2)^{n-1}(\cF_1\oplus\cF_2),\eps)$-decomposable system for $A\oplus A$, it follows that 
\[
r_\nuc\left(d,\bigcup_{j=0}^{n-1}(\alpha_1\oplus\alpha_2)^{j}(\cF_1\oplus\cF_2),\eps\right) \leq 2 \max_{i=1,2}r_\nuc\left(d,\bigcup_{j=0}^{n-1}\alpha_i^{j}(\cF_i),\eps\right).
\]Therefore, \[\h_\nuc(\alpha_1\oplus\alpha_2,\cF_1\oplus\cF_2,\eps)\leq \max_{i=1,2}\h_\nuc(\alpha_i,\cF_i,\eps)\leq\max_{i=1,2}\h_\nuc(\alpha_i).\]Since $\cF_1\oplus\cF_2$ and $\eps$ were arbitrary, we get that \[\h_\nuc(\alpha_1\oplus\alpha_2)\leq\max\{\h_\nuc(\alpha_1),\h_\nuc(\alpha_2)\}.\] The same argument shows that $\h_\dr(\alpha_1\oplus\alpha_2)=\max\{\h_\dr(\alpha_1),\h_\dr(\alpha_2)\}$.
\end{proof}

In the spirit of \cite[Proposition 2.8]{Brown99}, the notions of entropy from Definition \ref{defn: NucDimAndDrEntropy} are invariant under conjugacy.

\begin{prop}\label{prop: Conjugacy}
Let $A$ be a $\C$-algebra and $\alpha,\sigma$ be automorphisms of $A$. Then \[\h_\nuc(\sigma\circ\alpha\circ\sigma^{-1})=\h_\nuc(\alpha) \ \text{and} \ \h_\dr(\sigma\circ\alpha\circ\sigma^{-1})=\h_\dr(\alpha).\]
\end{prop}

\begin{proof}
If $\dim_\nuc(A)=\infty$, then $\h_\nuc(\alpha)=\h_\nuc(\sigma\circ\alpha\circ\sigma^{-1})=\infty$, so there is nothing to prove. Therefore, we will assume that $A$ has finite nuclear dimension, say $d$. Let $\cF\subseteq A$ be finite, $\eps>0$, and $n\in\mathbb{N}$. Then, by Lemma \ref{lemma: TechLemBrown}, we get that
\begin{align*}
r_\nuc\left(d,\bigcup_{j=0}^{n-1}(\sigma\circ\alpha\circ\sigma^{-1})^j(\cF),\eps\right)&=r_\nuc\left(d,\bigcup_{j=0}^{n-1}\sigma\circ\alpha^j\circ\sigma^{-1}(\cF),\eps\right)\\&=r_\nuc\left(d,\bigcup_{j=0}^{n-1}\alpha^j(\sigma^{-1}(\cF)),\eps\right).
\end{align*}
This implies that \[\h_\nuc(\sigma\circ\alpha\circ\sigma^{-1},\cF,\eps)=\h_\nuc(\alpha,\sigma^{-1}(\cF),\eps)\leq\h_\nuc(\alpha).\]Since $\cF$ and $\eps$ were arbitrary, we get that $\h_\nuc(\sigma\circ\alpha\circ\sigma^{-1})\leq \h_\nuc(\alpha)$. Conversely, if we let $\gamma=\sigma\circ\alpha\circ\sigma^{-1}$, then the proof above shows that\[\h_\nuc(\gamma)\geq\h_\nuc(\sigma^{-1}\circ\gamma\circ\sigma)=\h_\nuc(\alpha),\] so the conclusion follows. Clearly the same argument also shows that $\h_\dr(\sigma\circ\alpha\circ\sigma^{-1})=\h_\dr(\alpha)$. 
\end{proof}

\section{The commutative case} \label{sect: Commutative}

In this section, we will prove that the notions of coloured entropy we introduced, coincide with the usual topological entropy in the commutative setting. This is known to be the case for the Brown--Voiculescu entropy, as shown in \cite[Proposition 4.8]{Voiculescu95} when $X$ is a compact metric space and extended in \cite[Proposition 3.1]{KerrLi05} to the case of a general compact Hausdorff space. To prove a corresponding result for the coloured and the contractive coloured entropies, we first introduce a new variant of topological entropy.

Recall that a \emph{refinement} of an open cover $\cU$ of a topological space $X$ is an open cover $\cV$ such that every $V\in\cV$ is contained in some $U\in\cU$. The \emph{least common refinement} of open covers $\cU$ and $\cV$ is the set $\cU\vee\cV$ of nonempty intersections $U\cap V$ of elements $U\in\cU$, $V\in\cV$. Given $n\in\N$, we say that an open cover $\cU$ is \emph{$n$-coloured} if there are subsets $\cU_0,\dots,\cU_{n-1}$ of $\cU$ such that each $\cU_i$ is a collection of pairwise disjoint sets, and $\cU=\bigcup_{i=0}^{n-1}\cU_i$.

\begin{defn}\label{defn: ColEntr}
Let $X$ be a compact Hausdorff space with covering dimension equal to $d\in\zz_+$ and let $T$ be a homeomorphism of $X$. 
For any finite open cover $\cU$ of $X$ and $n\in\mathbb{N}$, let $\cU_0^{n-1}=\bigvee_{j=0}^{n-1}T^{-j}(\cU)$.
Then define $N_c(\cU_0^{n-1})$ to be the minimal cardinality of a $(d+1)$-coloured refinement of the open cover $\cU_0^{n-1}$. 
Finally, we define the \emph{coloured topological entropy} of $T$ to be \[h_{\tp,c}(T)=\sup\limits_{\cU}\lim_{n\to\infty}\frac{1}{n}\log N_c(\cU_0^{n-1}),\]where $\cU$ ranges over all finite open covers of $X$
\end{defn}

\begin{lemma}\label{lemma: UpperBddComm1}
Let $X$ be a nonempty compact Hausdorff space with finite covering dimension and $T$ be a homeomorphism of $X$. If we denote the induced automorphism of $C(X)$ by $\alpha_T$, then $\h_\dr(\alpha_T)\leq h_{\tp,c}(T)$.
\end{lemma}

\begin{proof} 
Let $d\in\mathbb{Z}_+$ be the covering dimension of the space $X$. The conclusion follows from the cpc approximations constructed in the proof of \cite[Proposition 4.8]{Voiculescu95}.
We will include the details for the convenience of the reader. 
Let $\cF$ be a finite subset of $C(X)$, $\eps>0$, $n\in\mathbb{N}$, and $\cU=\{U_1,\ldots,U_m\}$ be an open cover of $X$ such that \[|f(x)-f(y)|<\eps, \ \text{for all} \ f\in\cF,\] whenever $x,y\in U_j$ for any $1\leq j\leq m$. 
Let $\cV$ be a $(d+1)$-coloured refinement of $\cU_0^{n-1}$ with minimal cardinality (i.e., $|\cV|=N_c(\cU_0^{n-1}))$. 
In particular, note that for any $f\in \bigcup_{j=0}^{n-1}\alpha_T^j(\cF)$, $V\in\cV$, and $x,y\in V$, we have that $|f(x)-f(y)|<\eps$. 

We can now construct a suitable decomposable system for $C(X)$. 
Let $X_n=(x_j)_{j=1}^{N_c(\cU_0^{n-1})}$ be a subset of $X$ obtained by picking one element from each member of $\cV$, and let $(h_j)_{j=1}^{N_c(\cU_0^{n-1})}$ be a partition of unity subordinate to the open cover $\cV$.
We then define $\psi\colon C(X)\to C(X_n)$ by restriction, and $\varphi\colon C(X_n)\to C(X)$ by \[\varphi(g)=\sum_{j=1}^{N_c(\cU_0^{n-1})}g(x_j)h_j\in C(X).\] Exactly as in \cite[Proposition 4.8]{Voiculescu95}, one gets \[\|f-\varphi(\psi(f))\|_{\sup} \leq\eps, \ \text{for any} \ f\in\bigcup_{j=0}^{n-1}\alpha_T^j(\cF).\]Moreover, $\psi$ is cpc (in fact, a $^*$-homomorphism), and $\varphi$ is also cpc since $h_j\geq 0$ for any $j$ and $\sum h_j=1$. 
Finally, since $\cV$ is a $(d+1)$-coloured refinement, it is immediate to see that $(C(X_n),\psi,\varphi)$ is a contractive $(d,\cF\cup\alpha_T(\cF)\cup\ldots\cup\alpha_T^{n-1}(\cF),\eps)$-decomposable system for $C(X)$. Thus, \[\h_\dr(\alpha_T,\cF,\eps)\leq \lim\limits_{n\to\infty}\frac{1}{n}\log N_c(\cU_0^{n-1})\leq h_{\tp,c}(T),\]which implies that  $\h_\dr(\alpha_T)\leq h_{\tp,c}(T)$.
\end{proof}

\begin{lemma}\label{lemma: UpperBddComm2}
Let $X$ be a nonempty compact metric space with finite covering dimension and $T$ be a homeomorphism of $X$. Then $h_{\tp,c}(T)\leq h_\tp(T)$.    
\end{lemma}

\begin{proof}
Let $d\in\mathbb{Z}_+$ be the covering dimension of the space $X$ and let $\cV=\{V_1,\ldots,V_m\}$ be a finite open cover of $X$. By following the strategy in \cite[Theorem V.1]{HurewiczWallman48}, we claim that there exists a $(d+1)$-coloured refinement of $\cV$ with cardinality $(d+1)|\cV|$. First, we can apply the decomposition theorem \cite[Theorem III.3]{HurewiczWallman48} to write \[X=X_0\cup\ldots\cup X_d\] as a union of $(d+1)$ sets of dimension $0$. As $\cV$ induces an open cover for each of the $(d+1)$ sets, we can apply Proposition $B)$ on page $54$ of \cite{HurewiczWallman48} to find refinements for each $X_i$, where the open sets are disjoint. Moreover, Proposition $B)$ mentioned above also ensures that each refinement has cardinality $|\cV|=m$. Thus, the union of all these open sets is a $(d+1)$-coloured refinement of $\cV$ of cardinality $(d+1)|\cV|$.

We now let $\cU$ be an open cover of $X$ and $n\in\mathbb{N}$. Applying the result above to the open cover $\cU_0^{n-1}$ yields that \[N_c(\cU_0^{n-1})\leq (d+1)N(\cU_0^{n-1}),\] where $N(\cdot)$ denotes the minimal cardinality of a subcover. Therefore,
\begin{align*}
\lim_{n\to\infty}\frac{1}{n}\log N_c(\cU_0^{n-1}) &\leq \lim_{n\to\infty}\frac{1}{n}(\log (d+1)+\log N(\cU_0^{n-1}))\\
&=\lim_{n\to\infty}\frac{1}{n}\log N(\cU_0^{n-1}).
\end{align*}
By definition of the topological entropy $h_\tp(T)$ of $T$ and its coloured variant $h_{\tp,c}(T)$, this shows that $h_{\tp,c}(T)\leq h_\tp(T)$.
\end{proof}

Combining Lemma~\ref{lemma: UpperBddComm1} and Lemma~\ref{lemma: UpperBddComm2}, we can prove that the (contractive) coloured entropy coincides with the topological entropy in the commutative case.

\begin{theorem}\label{thm: CommCase}
Let $X$ be a nonempty compact metric space with finite covering dimension and $T$ be a homeomorphism of $X$. 
If we denote the induced automorphism of $C(X)$ by $\alpha_T$, then \[h_\tp(T)=\h(\alpha_T)=\h_\nuc(\alpha_T)=\h_\dr(\alpha_T)=h_{\tp,c}(T).\]    
\end{theorem}

\begin{proof}
Combining \cite[Proposition 4.8]{Voiculescu95}, Proposition~\ref{prop: IneqBrownVoicEntr}, Lemma~\ref{lemma: UpperBddComm1}, and Lemma~\ref{lemma: UpperBddComm2} in this order, we obtain the sequence of inequalities
\[h_\tp(T)=\h(\alpha_T)\leq\h_\nuc(\alpha_T)\leq\h_\dr(\alpha_T)\leq h_{\tp,c}(T)\leq h_\tp(T).\]Hence, we have equality everywhere and the conclusion follows.
\end{proof}

\section{A variational type principle for coloured entropy}\label{sect: VarPr}

In this section, we will study the contractive coloured entropy in the setting of simple classifiable $\C$-algebras. The main objective is obtaining a characterisation of the contractive coloured entropy in terms of a suitable notion of entropy for traces. This characterisation aims to shed light on the problem of obtaining a variational principle for noncommutative topological entropy. 

Recall that the \emph{variational principle} states that if $T$ is a homeomorphism of a compact Hausdorff space $X$, then $h_\tp(T)=\sup_{\mu}h_\mu(T),$ where the supremum is taken over $T$-invariant Borel probability measures $\mu$ of $X$ (see for example \cite[Section 6.3, Theorem 3.10]{Petersen83}). In the noncommutative case, Voiculescu showed in \cite[Proposition 4.6]{Voiculescu95} that for an automorphism $\alpha$ of a $\C$-algebra $A$, one has that $\h(\alpha)\geq\sup_\sigma h_\sigma(\alpha)$, where the supremum is taken over $\alpha$-invariant states $\sigma$. To consider this problem in the context of the contractive coloured entropy, we will introduce a new notion of entropy for \emph{quasidiagonal traces} (see Definition~\ref{defn: AmenQDTraces}).

\begin{defn}\label{defn: TraceEntropy}
Let $A$ be a $\C$-algebra, $\alpha$ be an endomorphism of $A$, and $\tau$ be a quasidiagonal trace on $A$. For any finite set $\cF\subseteq A$ and any $\eps>0$, we say $(F,\tau_F,\varphi)$ is an \emph{$(\cF,\eps)$-quasidiagonal system} for $\tau$ if $F$ is finite dimensional with a trace $\tau_F$, and $\varphi\colon A\to F$ is a cpc map such that \[\|\varphi(ab)-\varphi(a)\varphi(b)\|\leq \eps\]and \[|\tau(a)-\tau_F(\varphi(a))|\leq\eps\] for any $a,b\in\cF$. We then consider the \emph{$(\cF,\eps)$-quasidiagonal rank} of $\tau$ to be \[\rqd_\tau(\cF,\eps)=\inf\{\mathrm{rank}(F)\mid (F,\tau_F,\varphi) \ \text{is $(\cF,\eps)$-quasidiagonal} \ \text{for} \ \tau_F\in T(F)\}.\] 
We now further define the quantities
\begin{align*}
&\hqd_\tau(\alpha,\cF,\eps)=\limsup\limits_{n\to\infty}\frac{1}{n}\log \rqd_\tau(\cF\cup\alpha(\cF)\cup\ldots\cup\alpha^{n-1}(\cF),\eps),\\
&\hqd_\tau(\alpha,\cF)=\sup\limits_{\eps>0}\hqd_\tau(\alpha,\cF,\eps), \\
&\hqd_\tau(\alpha)=\sup\limits_{\cF}\hqd_\tau(\alpha,\cF),\ \text{with $\cF$ ranging over all finite subsets of $A$.}
\end{align*}We will often refer to $\hqd_\tau(\alpha)$ as the \emph{quasidiagonal entropy of $\alpha$ with respect to $\tau$}. 
\end{defn}

\begin{rmk}\label{rmk: AmenRankTau}
If $\tau$ is an amenable trace, one can define a similar notion of an amenable entropy of $\alpha$ with respect to $\tau$ simply by considering cpc maps which approximately commute in $2$-norm rather than in norm.  
\end{rmk}




\begin{prop}\label{prop: LowerBddQD}
Let $A$ be a simple, unital $\C$-algebra with finite decomposition rank, and let $\alpha$ be an endomorphism of $A$. If $\tau$ is a trace on $A$, then \[\h_\dr(\alpha)\geq \hqd_\tau(\alpha).\]
\end{prop}

\begin{proof}
Let $d\in\mathbb{Z}_+$ be the decomposition rank of $A$. We will show that any trace on $A$ is quasidiagonal (see \cite[Proposition 8.5]{BBSTW19}) while keeping track of the rank of the finite-dimensional $\C$-algebras used in the approximations. Let $1_A\in \cF\subseteq A$ be finite, $0<\eps<1$, $n\in\mathbb{N}$, and $(F,\psi,\varphi)$ be a contractive $(d,\cF\cup\alpha(\cF)\cup\ldots\cup\alpha^{n-1}(\cF),\eps/3)$-decomposable system for $A$. Since $A$ is unital, we can assume that $\psi$ is unital without altering the finite-dimensional $\C$-algebra $F$ (see, for example, \cite[Lemma 2.2.5]{BrownOzawa08}). Moreover, following the proof of \cite[Proposition 5.1]{KirchbergWinterDecompRank}, we see that by possibly removing some summands of $F$, one can further assume that $\psi$ is unital and approximately multiplicative on $\cF$ up to $\eps/3$. In fact, these observations show that we could have defined $\h_\dr(\alpha)$ by further assuming that the downward approximating maps are unital and approximately multiplicative.

Let $\tau\in T(A)$ and $\sigma=\tau\circ\varphi$. Since $\varphi$ is a sum of $d+1$ cpc order zero maps, $\sigma$ is a sum of $d+1$ positive tracial functionals (\cite[Corollary 4.4]{WinterZachOrderZero}). However, $1_A\in \cF$ and $\psi$ is unital, so $\|\varphi(1_F)-1_A\|\leq\eps/3$, which in turn yields $\|\sigma(1_F)-1\|\leq\eps/3$. Consider $\tilde{\sigma}=\frac{\sigma}{\|\sigma\|}$, which is a tracial state on $F$.
Since $\|\sigma\|\in[1-\eps/3,1]$, we get that for any $a\in\cF\cup\alpha(\cF)\cup\ldots\cup\alpha^{n-1}(\cF)$
\begin{align*}
|\tilde{\sigma}(\psi(a))-\tau(a)| &\leq \frac{1}{\|\sigma\|}|\tau(\varphi(\psi(a)))-\tau(a)|+\left(\frac{1}{\|\sigma\|}-1\right)|\tau(a)|\\
&\leq \frac{2\eps/3}{1-\eps/3}=\frac{2\eps}{3-\eps} <\eps.
\end{align*}
This shows that $(F,\tilde{\sigma},\psi)$ is an $(\cF\cup\alpha(\cF)\cup\ldots\cup\alpha^{n-1}(\cF),\eps)$-quasidiagonal system for $\tau$, and hence \[r_\dr(d,\cF,\eps/3)\geq \rqd_\tau(\cF,\eps).\]Since $\cF$ and $\eps$ were arbitrary, it follows that $\h_\dr(\alpha)\geq \hqd_\tau(\alpha)$.
\end{proof}


In the following theorem, $\Z$ stands for the Jiang-Su algebra introduced in \cite{JiangSu99}. Moreover, a $\C$-algebra $A$ is said to be $\Z$-stable if $A\cong A\otimes\Z$.
Before proving the theorem, let us record a particular case of \cite[Theorem~5.5]{BBSTW19}, recasting it from a sequential version to a local approximation version.

\begin{theorem}[cf. \ {\cite[Theorem~5.5]{BBSTW19}}]\label{thm: RecastingBBSTW}
Let A be a simple, separable, unital, finite, nuclear, $\mathcal{Z}$-stable $\C$-algebra.
Let $\omega$ be a fixed free ultrafilter on $\mathbb{N}$, $\phi_1\colon A\to A_\omega$ be a unital $*$-homomorphism induced by a sequence of maps $(\phi_1^{(k)})_{k\in\mathbb{N}}$ and $\phi_2\colon A\to A_\omega$ be a cpc order zero map induced by a sequence of maps $(\phi_2^{(k)})_{k\in\mathbb{N}}$ such that \[\tau\circ\phi_1=\tau\circ\phi_2^m, \ \tau\in T(A_\omega), \ m\in\mathbb{N},\]
where order zero functional calculus is used to interpret $\phi_2^m$.\footnote{Note that since $A$ is simple and $\phi_1$ is unital, $\phi_1$ is totally full in the sense of \cite[Definition~1.1]{BBSTW19}. In particular, the statement above is indeed a particular case of \cite[Theorem~5.5]{BBSTW19}.}
Let $h\in\mathcal{Z}_+$ be a positive contraction with spectrum $[0,1]$.
Then there exists a sequence of unitaries $(u_n)_{n\in\mathbb{N}}\subset A\otimes\mathcal{Z}$ such that for any $\varepsilon>0$ and any finite set $\cF$ of $A$, there exists $k\in\mathbb{N}$ such that \[\|(\phi_1^{(n)}(a)\otimes h) -u_n(\phi_2^{(n)}(a)\otimes h)u_n^*\|\leq\varepsilon,\] for all $a\in\cF$ and any $n\geq k$.
\end{theorem}

\begin{theorem}\label{thm: VariationalQDTraces}
Let $A$ be a simple, separable, unital, nuclear, $\mathcal{Z}$-stable $\C$-algebra with finitely many extremal traces and which satisfies the UCT. Suppose further that $A$ is not an AF-algebra. If $\alpha$ is an endomorphism of $A$, then \[\h_\dr(\alpha)=\sup_{\tau\in\partial_e T(A)}\hqd_\tau(\alpha)=\sup_{\tau\in T(A)}\hqd_\tau(\alpha).\]
\end{theorem}

\begin{proof}
Combining \cite[Theorem I, Theorem H]{CPOU21} with \cite[Theorem A]{TikuisisWhiteWinter17}, we get that $A$ has decomposition rank at most $1$. Since $A$ is not AF, it follows that $\dr(A)=1$ (\cite[Example 4.1]{KirchbergWinterDecompRank}). If we show the first equality, combining it with Proposition~\ref{prop: LowerBddQD}, the second equality in the statement of the theorem follows. Thus, using Proposition~\ref{prop: LowerBddQD} again, it remains to show that \[\h_\dr(\alpha)\leq \sup_{\tau\in\partial_e T(A)}\hqd_\tau(\alpha).\] We will follow the strategy in \cite[Lemma 7.4, Theorem 7.5]{BBSTW19}, where the quasidiagonal approximations of the finitely many extremal traces are used to produce a contractive decomposable system for $A$. Let $\tau_1,\ldots,\tau_m$ be the extremal traces of $A$. Then each $\tau_i$ is quasidiagonal by \cite[Theorem A]{TikuisisWhiteWinter17}, so let $(F_{i,k},\tau_{i,k},\theta_{i,k})_{k\in\mathbb{N}}$ be a quasidiagonal system for $\tau_i$. Since $A$ is unital, we can assume that each map $\theta_{i,k}$ is unital. In fact, we can choose the sequence $(F_{i,k},\tau_{i,k},\theta_{i,k})_{k\in\mathbb{N}}$ in a more precise way. Let $(\cG_k)_{k\in\mathbb{N}}$ be an increasing sequence of finite sets whose union is dense in $A$, and $n\in\mathbb{N}$. Then for each $k\in\mathbb{N}$, we choose $(F_{i,k},\tau_{i,k},\theta_{i,k})$ to be a $(\cG_k\cup\alpha(\cG_k)\cup\ldots\cup\alpha^{n-1}(\cG_k),1/k)$-quasidiagonal system for $\tau_i$.

Let $\omega$ be a fixed free ultrafilter on $\mathbb{N}$. Following \cite[Lemma 7.4]{BBSTW19}, we will first find a sequence of finite-dimensional $\C$-algebra $F_k$, together with sequences of cpc maps $\theta_k\colon A\to F_k$ and $\eta_k\colon F_k\to A$ such that $\eta_k$ is order zero, 
\begin{equation}\label{eq: Approx1}
\lim_{k\to\omega}\|\theta_k(x)\theta_k(y)\|=0,\ \text{if}\ xy=0,
\end{equation}and if $\Phi\colon A\to A_\omega$ is the map induced by the sequence $\eta_k\circ\theta_k$, then
\begin{equation}\label{eq: Approx2}
\tau_i\circ\Phi=\tau_i, \ 1\leq i\leq m.
\end{equation} Here we abuse notation and we also denote the unique extension of each $\tau_i$ to $A_\omega$ by $\tau_i$.

Noticing that $C(\partial_e T(A))\cong \mathbb{C}^m$, as in the proof of \cite[Lemma 7.4]{BBSTW19}, using the functions $f_i(\tau_j)=\delta_{i,j}$ for $1\leq i,j\leq m$, we can find pairwise orthogonal sequences of positive contractions $(d_{1,k})_{k=1}^\infty,\ldots,(d_{m,k})_{k=1}^\infty\in \ell^\infty(A)$ such that
\begin{equation}\label{eq: Bauer3}
\lim_{k\to\omega}\max_{1\leq j\leq m}|\tau_j(d_{i,k})-\delta_{i,j}|=0,\ 1\leq i\leq m.
\end{equation}

We need one more ingredient to obtain a sequence of maps $\theta_k$ satisfying \eqref{eq: Approx1}. Let $\mathcal{Q}$ be the universal UHF-algebra. Then let $\Theta_i\colon A\to \mathcal{Q}_\omega$ be the $^*$-homomorphism induced by the quasidiagonal system $(F_{i,k},\tau_{i,k},\theta_{i,k})_{k\in\mathbb{N}}$. As in \cite[Lemma 7.4]{BBSTW19}, we can use \cite[Lemma 6.1]{SatoWhiteWinter15} to produce a sequence of cpc order zero maps $\psi_k\colon\mathcal{Q}\to\mathcal{Z}$ such that
\begin{equation}\label{eq: Bauer4}
\tau_{\mathcal{Z}_\omega}\circ\Psi\circ\Theta_i=\tau_i, \ 1\leq i\leq m,
\end{equation}where $\Psi$ is the cpc order zero map induced by the sequence $\psi_k$.

We now let $F_k=\bigoplus_{i=1}^mF_{i,k}$, $\theta_k=\bigoplus_{i=1}^m\theta_{i,k}\colon A\to F_k$ as given by the chosen quasidiagonal systems, and $\tilde{\eta}_k\colon F_k\to A\otimes\mathcal{Z}$ defined by
\begin{equation}\label{eq: Bauer6}
\tilde{\eta}_k(y_1,\ldots,y_m)=\sum_{i=1}^m d_{i,k}\otimes\psi_k(y_i),\ y_i\in F_{i,k}.    
\end{equation}
Since the sequence $\theta_{i,k}$ yields a quasidiagonal system for $\tau_i$, the map $\theta_k$ is unital and it satisfies \eqref{eq: Approx1}. Moreover, since $\psi_k$ is cpc order zero and the elements $d_{1,k},\ldots, d_{m,k}$ are pairwise orthogonal, it follows that $\tilde{\eta}_k$ is a cpc order zero map. To obtain maps $\eta_k$ satisfying \eqref{eq: Approx2}, we compose $\tilde{\eta}_k$ with an isomorphism $\beta\colon A\otimes\mathcal{Z}\to A$ such that $\tau\circ\beta=\tau\otimes\tau_{\mathcal{Z}}$ for any $\tau\in T(A)$. Consider the order zero map $\Phi$ induced by the sequence $\eta_k\circ\theta_k$.

By \eqref{eq: Bauer3} and \eqref{eq: Bauer4}, it follows that $\Phi$ satisfies \eqref{eq: Approx2}. Precisely, for each extremal trace $\tau_j$,
\begin{align*}
\tau_j\circ\eta_k\circ\theta_k &=(\tau_j\otimes\tau_{\mathcal{Z}})\left(\sum_{i=1}^m d_{i,k}\otimes(\psi_k\circ\theta_{i,k})\right)\\
&= \sum_{i=1}^m\tau_j(d_{i,k})(\tau_{\mathcal{Z}}\circ\psi_k\circ\theta_{i,k}).
\end{align*}Sending $k\to\omega$, it is immediate that \eqref{eq: Bauer3} and \eqref{eq: Bauer4} yield the desired equality.

We can now keep track of the sequence of finite-dimensional $\C$-algebras $F_k$ in the proof of \cite[Theorem 7.5]{BBSTW19}. Let $h$ be a positive contraction in $\mathcal{Z}$ which has spectrum $[0, 1]$, $\cF\subseteq A$ be finite, $\eps>0$, and $n\in\mathbb{N}$. Write $M=\max\{\|a\|\mid a\in\cF\}+\eps+1$. Then there exists $N\in\mathbb{N}$ such that $\cF$ is $\eps/5M$-contained in $\cG_k$ and $1/k\leq \eps/5M$ for any $k\geq N$. In particular, $(F_{i,k},\tau_{i,k},\theta_{i,k})$ is an $(\cF\cup\alpha(\cF)\cup\ldots\cup\alpha^{n-1}(\cF),\eps)$-quasidiagonal system for $\tau_i$ for each $1\leq i\leq m$ and any $k\geq N$.

Then, we can apply Theorem~\ref{thm: RecastingBBSTW} to $\Phi$ and the canonical diagonal inclusion $A\hookrightarrow A_\omega$, once with $h$ and once with $1_{\mathcal{Z}}- h$ to get unitaries
$u^{(0)}, u^{(1)}\in A \otimes \mathcal{Z}$ and $k\in\mathbb{N}$, possibly greater than $N$, such that for all $x \in \cF\cup\alpha(\cF)\cup\ldots\cup\alpha^{n-1}(\cF)$, we have
\begin{equation}\label{eq: CPOU4}
\|x \otimes 1_{\mathcal{Z}} -
(u^{(0)}(\eta_k\circ\theta_k(x) \otimes h)u^{(0)*} + u^{(1)}(\eta_k\circ\theta_k(x) \otimes (1_{\mathcal{Z}}- h))u^{(1)*})\|\leq\eps.
\end{equation} 

We are now in the position to produce a decomposition rank approximation for the map $\id_A\otimes 1\colon A\to A\otimes\mathcal{Z}$. Consider the cpc map $\theta_k\oplus\theta_k\colon A\to F_k\oplus F_k$ and $\psi_k\colon F_k\oplus F_k\to A\otimes\mathcal{Z}$ given by \[\psi_k(x,y)=u^{(0)}(\eta_k(x) \otimes h)u^{(0)*} + u^{(1)}(\eta_k(y) \otimes (1_{\mathcal{Z}}- h))u^{(1)*}.\]Note that $\psi_k$ is a sum of two cpc order zero maps. Moreover, since $\theta_k$ is unital, $\psi_k(1_F\oplus 1_F)$ approximates $1_A\otimes 1_{\mathcal{Z}}$. Therefore $\psi_k$ is approximately contractive, so it can be made contractive by rescaling. Since $\mathcal{Z}$ is strongly self-absorbing and $A$ is $\mathcal{Z}$-stable, one can further compose with an isomorphism from $A\otimes \mathcal{Z}$ to $A$ to obtain a finite decomposition rank approximation for $A$. More precisely, by \cite[Lemma 4.4]{Rørdam04} there exists an isomorphism $\Xi\colon A\otimes\mathcal{Z}\to A$ such that
\begin{equation}\label{eq: Approx4}
\|x-\Xi(x\otimes 1_{\mathcal{Z}})\|\leq\eps, \ x\in\cF\cup\alpha(\cF)\cup\ldots\cup\alpha^{n-1}(\cF).
\end{equation}Then, composing $\psi_k$ with $\Xi$ and using \eqref{eq: CPOU4}, we obtain a finite decomposition rank approximation for $A$ on the finite set $\cF\cup\alpha(\cF)\cup\ldots\cup\alpha^{n-1}(\cF)$, up to $2\eps$.

All in all, this shows that \[r_\dr\left (1,\bigcup_{j=0}^{n-1}\alpha^j(\cF),2\eps\right)\leq2\cdot\mathrm{rank}(F_k)\leq 2m\max_{1\leq i\leq m}\mathrm{rank}(F_{i,k}).\] Therefore, we get
\[r_\dr\left(1,\bigcup_{j=0}^{n-1}\alpha^j(\cF),2\eps\right)\leq 2m\max_{1\leq i\leq m}\rqd_{\tau_i}\left(\bigcup_{j=0}^{n-1}\alpha^j(\cG_k),1/k\right),\]which yields that
\[\h_\dr(\alpha,\cF,2\eps)\leq \max_{1\leq i\leq m}\hqd_{\tau_i}(\alpha,\cG_k,1/k)\leq\max_{1\leq i\leq m}\hqd_{\tau_i}(\alpha).\] Since $\cF$ and $\eps$ were arbitrary, we get that \[\h_\dr(\alpha)\leq\max_{1\leq i\leq m}\hqd_{\tau_i}(\alpha),\]which finishes the proof. 
\end{proof}

\begin{rmk}
Note that one cannot use the strategy in \cite[Lemma 7.4]{BBSTW19} or \cite[Lemma 5.2]{CPOU21} to extend the result in Theorem \ref{thm: VariationalQDTraces} to more general trace simplices. The reason is that in both situations, one ends up using Kirchberg's $\eps$-test to reindex. This alters the finite-dimensional $\C$-algebras in the desired cpc approximations. 
\end{rmk}

\begin{rmk}\label{rmk: DiffDefEntropy}
To avoid making the extra assumption that the $\C$-algebra is not AF, one could modify the definition of the decomposition rank entropy by considering a quantity $\h_{\dr,d}$ for each $d$ such that $d\geq \dr(A)$. Precisely, if $d\geq \dr(A)$, one can find $(d+1)$-coloured contractive decomposable systems, so $\h_{\dr,d}$ can be defined exactly in the same fashion as the contractive coloured entropy using these systems. With this modification, Theorem~\ref{thm: VariationalQDTraces} shows that \[\h_{\dr,1}(\alpha)=\sup_{\tau\in\partial_e T(A)}\hqd_\tau(\alpha)=\sup_{\tau\in T(A)}\hqd_\tau(\alpha)\] for each classifiable $\C$-algebra $A$ with finitely many extremal traces, and any endomorphism $\alpha$ of $A$.

In general, if $\alpha$ is an endomorphism of a unital $\C$-algebra $A$ and $d\ge\dr(A)$, then $\h(\alpha) \le \h_{\dr,d+1}(\alpha) \le \h_{\dr,d} \le \h_\dr(\alpha)$. Therefore, these numbers all coincide whenever $\h(\alpha)=\h_\dr(\alpha)$ (which is the case, for example, in Proposition~\ref{prop: Shift}, Theorem~\ref{thm: CommCase}, Corollary~\ref{cor: DenseZeroEntropy}, Theorem~\ref{thm: DenseInfEntropy} and Theorem~\ref{thm: GenericInfEntropy}). Note finally that a similar modification could be made to the definition of $\h_\nuc$, and similar comments would apply.
\end{rmk}

\section{Typical values of entropy} \label{sect: TypicalEntropy}

In this section, we investigate typical values of noncommutative entropy among endomorphisms and automorphisms of $\Z$-stable $\C$-algebras, particularly those that are classifiable. By `typical' we mean relative to the topology of pointwise convergence. More precisely, for a separable, unital $\C$-algebra $A$ we equip $\End(A)$ with the completely metrisable topology provided by the family of pseudometrics
\[
\{d_{\cF}(\alpha,\beta) \mid \cF\subseteq A \text{ finite}\},
\]
where
\[
d_{\cF}(\alpha,\beta)=\max_{a\in \cF}\|\alpha(a)-\beta(a)\|.
\]
To ensure that the topology on $\Aut(A)$ is also completely metrisable, we use the family
\[
\{d_{\cF}(\alpha,\beta)+d'_{\cF}(\alpha^{-1},\beta^{-1}) \mid \cF\subseteq A \text{ finite}\},
\]
where
\[
d'_{\cF}(\varphi,\psi) = \inf_{u\in\cU(A)}\max_{a\in \cF}\|u\varphi(a)u^*-\psi(a)\|.
\]
In the commutative setting, complete metrics for these topologies are given by
\[
\rho(g,h) = \sup_{x\in X}d(g(x),h(x))
\]
in the space of continuous maps from a compact metric space $(X,d)$ to itself, and by
\[
\rho'(g,h) = \sup_{x\in X}d(g(x),h(x))+d(g^{-1}(x),h^{-1}(x))
\]
in the space $\home(X)$ of self-homeomorphisms of $(X,d)$.

First, we examine the prevalence of zero entropy among automorphisms of $\C$-algebras $A$ of real rank zero (as defined in \cite{BrownPedersen91}). We write $\cU_0(A)$ for the connected component of the identity in the unitary group $\cU(A)$ of $A$, $\mathrm{Inn}_0(A)$ for the corresponding set $\{\Ad_u \mid u\in\cU_0(A)\}$ of inner automorphisms and $\overline{\mathrm{Inn}_0(A)}$ for the closure of $\mathrm{Inn}_0(A)$ in $\Aut(A)$.

\begin{prop} \label{prop: RR0}
Let $A$ be a unital, exact $\C$-algebra of real rank zero. Then, $\h=0$ on a dense subset of $\overline{\mathrm{Inn}_0(A)}$. The same is true for $\h_\nuc$ if $\dimnuc(A)<\infty$, and for $\h_\dr$ if $\dr(A)<\infty$.
\end{prop}

\begin{proof}
Let $u\in\cU_0(A)$, $\cF\subseteq A$ be finite and $\eps>0$. Write $M=1+\max\{\|a\| \mid a\in \cF\}$. Since $A$ has real rank zero, there exists $v\in\cU_0(A)$ with finite spectrum (which we may assume consists of rational points on the circle) such that $\|u-v\|<\frac{\eps}{4M}$ (see \cite[Corollary 6]{Lin93}). Then,
\begin{align*}
d_{\cF}(\Ad_u,\Ad_v)+d'_{\cF}(\Ad_u^{-1},\Ad_v^{-1}) &\le \max_{a\in \cF}(\|uau^*-vav^*\| +\|u^*au-v^*av\|)\\
&<\eps
\end{align*}
and $\Ad_v$ has zero entropy (in any of the three senses, as appropriate) by virtue of being periodic.
\end{proof}

If $A$ is a UHF-algebra or the Cuntz algebra $\cO_2$, then the real rank of $A$ is zero and $\overline{\mathrm{Inn}_0(A)}=\Aut(A)$. We therefore have the following consequence of Proposition~\ref{prop: RR0} (cf.\ \cite[Proposition 8.7]{KerrLi05}).

\begin{cor} \label{cor: DenseZeroEntropy}
If $A$ is a UHF-algebra or the Cuntz algebra $\cO_2$, then there is a dense subset of $\Aut(A)$ on which $\h_\dr=\h_\nuc=\h=0$ in the former case, and $\h_\nuc=\h=0$ in the latter.
\end{cor}

In \cite{Kerr07}, Kerr asserts that infinite Brown--Voiculescu entropy occurs on a dense subset of $\Aut(A)$ for any unital, $\Z$-stable $\C$-algebra $A$. By Proposition~\ref{prop: IneqBrownVoicEntr}, it follows that the (contractive) coloured entropy exhibits the same behaviour. Kerr's observation is proved in exactly the same way as \cite[Proposition 8.9]{KerrLi05}, which is devoted to the `contractive approximation entropy' developed in that article. For completeness, we recall the full argument here. As a starting point, we need some way of establishing lower bounds for noncommutative entropy.

\begin{lemma}[\cite{Kerr04}] \label{lemma: KerrLowerBound}
There exists a universal constant $a>0$ with the following property. Let $A$ be a unital $\C$-algebra and let $\alpha\in\End(A)$. Suppose that there exist $K>0$ and a finite subset $\Upsilon=\{\upsilon_1,\dots,\upsilon_m\}$ of the unit ball of $A$ such that for every $n\in\N$, the map $\ell_1^{mn}(\cc) \to A$ that sends the standard basis to $\Upsilon \cup \alpha\Upsilon \cup \dots \cup \alpha^{n-1}\Upsilon$ is a (necessarily contractive) linear isomorphism onto its image whose inverse has norm at most $K$ (in other words, $\Span \Upsilon \cup \alpha\Upsilon \cup \dots \cup \alpha^{n-1}\Upsilon$ is $K$-equivalent to $\ell_1^{mn}(\cc)$). Then, $\h(\alpha)\ge a^{-1}K^{-2}m$.
\end{lemma}

\begin{proof}
As explained in \cite[Remark 3.10]{Kerr04}, this is a consequence of \cite[Lemma 3.1]{Kerr04}.
\end{proof}

\begin{prop}[\cite{KerrLi05}] \label{prop: TypeI}
If $A$ is a separable, unital $\C$-algebra such that $\h(\Ad_u)<\infty$ for every unitary $u\in A$, then $A$ is type $\mathrm{I}$.
\end{prop}

\begin{proof}
This is the $(4)\Rightarrow(1)$ implication of \cite[Theorem 6.4]{KerrLi05}, whose contrapositive is proved using the construction of \cite{Brown04}. Suppose that $A$ is not type $\mathrm{I}$. As discussed in the proof of \cite[Theorem 1.2]{Brown04}, theorems of Glimm and Voiculescu provide a unital subalgebra $B\subseteq A$, a unitary $u\in B$ and a surjective $^*$-homomorphism $\pi \colon B \to (\mathbb{M}_{2^\infty})^{\otimes\zz}$ such that $\pi\circ\Ad_u=\alpha\circ\pi$, where $\alpha\in\Aut((\mathbb{M}_{2^\infty})^{\otimes\zz})$ is the tensor shift. Let $m\in\N$. As observed in the proof of Proposition~\ref{prop: Shift}, $\alpha$ restricts to the tensor shift on $\alpha\in\Aut(\mathbb{M}_{2^m}(\cc)^{\otimes\zz})$ and hence to the commutative Bernoulli shift on the diagonal subalgebra $C(X^\zz)$, where $X=\{0,1\}^{\{1,\dots,m\}}$. As in the proof of \cite[Lemma 2.2]{Kerr07}, there exists a finite subset $\Omega=\{f_1,\dots,f_m\}$ of norm-one, self-adjoint elements of $C(X^\zz)$ such that for every $n\in\N$, $\Span \Omega\cup\alpha(\Omega)\cup\dots\cup\alpha^{n-1}(\Omega)$ is $2$-equivalent to $\ell_1^{mn}(\rr)$, hence $4$-equivalent to $\ell_1^{mn}(\cc)$. For each $j\in\{1,\dots,m\}$, choose $b_j\in B$ with $\|b_j\|=1$ such that $\pi(b_j)=f_j$. Writing $\Upsilon=\{b_1,\dots,b_m\}$, we then have that $\Span \Upsilon\cup\Ad_u(\Upsilon)\cup\dots\cup\Ad_u^{n-1}(\Upsilon)$ is also $4$-equivalent to $\ell_1^{mn}(\cc)$. Since $m\in\N$ is arbitrary, it follows from Lemma~\ref{lemma: KerrLowerBound} that $\h(\Ad_u)=\infty$.
\end{proof}

\begin{theorem}[\cite{KerrLi05}] \label{thm: DenseInfEntropy}
If $A$ is a unital, $\Z$-stable $\C$-algebra, then $\h_\dr=\h_\nuc=\h=\infty$ on dense subsets of $\End(A)$ and $\Aut(A)$.
\end{theorem}

\begin{proof}
Fix $\eps>0$, a finite set $\cF\subseteq A$ and $\alpha\colon A\to A$ an endomorphism or an automorphism. Since $\Z$ is not type $\mathrm{I}$, Proposition~\ref{prop: TypeI} provides a unitary $u\in\Z$ with $\h(\Ad_u)=\infty$. As in \cite[Lemma 8.8]{KerrLi05}, there is an isomorphism $\Phi\colon A \to A\otimes\Z$ such that $\|\Phi(x)-x\otimes1_\Z\|<\eps/2$ for every $x\in\cF\cup\alpha(\cF)$ (or $x\in\cF\cup\alpha(\cF)\cup\alpha^{-1}(\cF)$ if $\alpha$ is an automorphism). Then, as in the proof of \cite[Proposition 8.9]{KerrLi05}, the map $\beta=\Phi^{-1}\circ(\alpha\otimes\Ad_u)\circ\Phi$ is an endomorphism (which is invertible if $\alpha$ is) with $\h(\beta)=\infty$ (by monotonicity; \cite[Proposition 2.1]{Brown99}) such that $\|\beta(x)-\alpha(x)\|<\eps$ (and also $\|\beta^{-1}(x)-\alpha^{-1}(x)\|<\eps$ if $\alpha$ is an automorphism) for every $x\in\cF$.
\end{proof}

The question of whether infinite entropy is a \emph{generic} property (that is, whether the set of endomorphisms or automorphisms with infinite entropy contains a dense $G_\delta$ set) is more difficult. Our next goal is to extend the work of \cite{Kerr07}, which shows that infinite entropy occurs generically in the automorphism group of a specially constructed AF-algebra. Using the tools provided by modern classification, we show that this result can be extended to the kinds of $\C$-algebras studied in \cite{Jacelon22}, which can be thought of as noncommutative spaces of observables of topological manifolds.

In the following theorem and its proof, $\Aff(T(A))$ denotes the set of continuous affine maps $T(A)\to\rr$, and $\hat{}$ denotes the evaluation map $A_{sa}\to\Aff(T(A))$ (that is, $\hat{a}(\tau)=\tau(a)$ for every self-adjoint element $a\in A$ and trace $\tau\in T(A)$). The proof appeals to Yano's theorem \cite[Corollary 1.1]{Yano80}, which establishes infinite entropy as a generic property among homeomorphisms of a compact topological manifold $X$, provided that the dimension of $X$ is greater than one. We note that this dimensional restriction is necessary, as the topological entropy of any homeomorphism $f$ of a one-dimensional compact manifold $X$ is zero. Since in this case $f$ is described up to permutation of connected components by self-homeomorphisms of intervals and circles, this fact follows from \cite[Theorem 7.14 and Corollary 7.14.1]{Walters82}. On the other hand, the genericity result for \emph{noninvertible} maps $f\colon X\to X$ (namely, \cite[Corollary 2.1]{Yano80}) holds in all dimensions.

\begin{theorem}\label{thm: GenericInfEntropy}
Let $A$ be a simple, separable, unital $\C$-algebra that has finite nuclear dimension and satisfies the UCT. Suppose also that the extreme boundary of $T(A)$ is homeomorphic to a compact topological manifold (with or without boundary), and that $(K_0(A),K_0(A)_+,[1_A])$ admits a unique state (equivalently, $\hat{p}\in\Aff(T(A))$ is constant for every projection $p$ over $A$). Then, $\h_\dr=\h_\nuc=\h=\infty$ on a dense $G_\delta$ subset of the set of unital endomorphisms $\alpha$ of $A$ for which $T(\alpha)$ preserves $\partial_e T(A)$. If in addition the dimension of $\partial_e T(A)$ is at least two, then $\h_\dr=\h_\nuc=\h=\infty$ on a dense $G_\delta$ subset of $\Aut(A)$.
\end{theorem}

\begin{proof}
Let us first consider the case of automorphisms. By \cite[Theorem 1]{Yano80}, there is a dense $G_\delta$ subset $\cV$ of infinite-entropy homeomorphisms of $X:=\partial_e T(A)$, constructed as the intersection $\cV=\bigcap_{n=1}^\infty\cH_n$ of dense open sets $\cH_n \subseteq \home(X)$. Elements $f\in\cH_n$ have the property that for some $k=k(f,n)\in\N$, $f^k$ is conjugate to a `pseudo horseshoe of type $n^k$'. The precise meaning of this is not as important as the observation made in the proof of \cite[Proposition 2]{Yano80}, namely that for every $f\in\cH_n$ there is a closed invariant subset $Y\subseteq X$ that factors onto the full shift on an alphabet of $n^k$ symbols.

We claim that elements of $T^{-1}(\cV)=\{\alpha \in \Aut(A) \mid T(\alpha)\in\cV\}$ have infinite Brown--Voiculescu entropy. To prove this, we will show that for an automorphism $\alpha$ with $f:=T(\alpha)\in\cV$, and for every $q\in\N$, $\h(\alpha)$ is bounded below by $ckq$ for some universal constant $c$. Since $q$ is arbitrary, this will imply that $\h(\alpha)=\infty$.

We proceed as in \cite{Kerr07}. Since $f=T(\alpha)\in\cH_{2^q}$, it follows from \cite[Lemma 2.2]{Kerr07} that there is a set $\Omega$ of norm-one, self-adjoint elements of $C(X)$ with $|\Omega|=2^{kq}$ (where $k=k(f,2^q)$) such that for every $n\in\N$, $\Span \Omega \cup f^*\Omega \cup \dots \cup (f^*)^{n-1}\Omega$ is $2$-equivalent to $\ell_1^{nkq}(\rr)$ (in the sense of Lemma~\ref{lemma: KerrLowerBound}, and where $f^*\colon C(X)\to C(X)$ is the induced endomorphism $g\mapsto g\circ f$). By \cite[Proposition 2.1]{CGSTW23}, we can find for each $g\in\Omega$ a self-adjoint, norm-one element $c_g\in A$ such that $\|\hat{c_g}|_{\partial_e T(A)}-g\| < 1/8$. (First approximate $g$ within $1/16$ by a cut-down $g'$ of norm $15/16$, then apply \cite[Proposition 2.1(i)]{CGSTW23} to $g'$ with $\eps=1/16$ to obtain a self-adjoint contraction $b_g\in A$ with $\hat{b_g}=g'$. Then, $c_g=b_g/\|b_g\|$ does the job.) With $\Upsilon:=\{c_g \mid g\in\Omega\}$, it follows exactly as in the proof of \cite[Lemma 2.3]{Kerr07} that $\Span \Upsilon \cup \alpha\Upsilon \cup \dots \cup \alpha^{n-1}\Upsilon$ is $4$-equivalent to $\ell_1^{nkq}(\rr)$, hence $8$-equivalent to $\ell_1^{nkq}(\cc)$. (In \cite{Kerr07}, the elements $c_g$ are projections, but the inequalities that establish \cite[Lemma 2.3]{Kerr07} are valid more generally for contractions.) Lemma~\ref{lemma: KerrLowerBound} then gives $\h(\alpha)\ge a^{-1}8^{-2}kq$, proving the claim.

Finally, \cite[Theorem 3.4]{Jacelon22} (together with the observation \cite[Remark 3.6]{Jacelon22} that the framework of \cite{CGSTW23} can be used to remove the hypothesis of torsion-free $K_1$) implies that $T^{-1}(\cV)$ is a dense $G_\delta$ subset of $\Aut(A)$. Since \cite[Theorem 3.4]{Jacelon22} also applies to $\partial_e T(A)$-preserving endomorphisms of $A$, and \cite[Theorem 2]{Yano80} and the proof of \cite[Proposition 2']{Yano80} apply to the space of continuous maps $X\to X$, the argument in the non-invertible setting is exactly the same. 
\end{proof}

\section{Open questions} \label{sect: Questions}
\renewcommand{\labelenumi}{(\arabic{enumi})}

We will end with a collection of open questions.

\begin{enumerate}
   \item Is there a $\C$-algebra $A$ with finite decomposition rank and an endomorphism $\alpha$ of $A$ for which either or both of the inequalities $\h(\alpha)\leq\h_\nuc(\alpha)\leq\h_\dr(\alpha)$ is strict?

   \item The definition of $\h_\nuc(\alpha)$ also makes sense for cpc maps $\alpha\colon B\to B$ that need not be $^*$-homomorphisms. If $A$ is an irreducible, non-permutation $\{0,1\}$-valued matrix and $\theta_A$ is the canonical cpc endomorphism of the Cuntz--Krieger algebra $\cO_A$ discussed in Remark~\ref{rmk: CuntzKrieger}, then is $\h_\nuc(\theta_A)$ equal to the logarithm of the spectral radius of $A$?
   
    \item In \cite[Theorem 3.5]{Brown99}, Brown showed the following result. If $\alpha\colon G \to \Aut(A)$ is an action of a discrete abelian group $G$ on a unital, exact $\C$-algebra $A$, and $\lambda_g \in A\rtimes_\alpha G$ denotes the unitary implementing the automorphism $\alpha_g\in \Aut(A)$, then \[\h_A(\alpha_g) = \h_{A\rtimes_\alpha G}(\Ad(\lambda_g)).\] Introduced in \cite{HirshbergWinterZacharias} and extended for example in \cite{SzaboWuZacharias,Szabo19}, finite Rokhlin dimension is a regularity property for group actions on $\C$-algebras that in many cases entails $\Z$-stability or finite nuclear dimension of the crossed product $A\rtimes_\alpha G$ if these properties hold for $A$. Does the equality above hold for $\h_\nuc$ (or $\h_\dr$) if one further assumes that $A$ has finite nuclear dimension (or finite decomposition rank) and that, in the sense of \cite{Szabo19}, the action $\alpha$ has finite Rokhlin dimension with commuting towers?\footnote{One can ask the same question if $\alpha$ has finite Rokhlin dimension (without commuting towers) if $G$ is a group for which this is defined.}
    \item Can Theorem~\ref{thm: VariationalQDTraces} be extended to arbitrary trace spaces? Moreover, using a version for amenable entropy instead (see Remark~\ref{rmk: AmenRankTau}), is there a variational principle for the coloured entropy $(\h_\nuc$)?
\end{enumerate}

\bibliographystyle{abbrv}
\bibliography{entropy}

@article {WinterZachOrderZero,
    AUTHOR = {Winter, Wilhelm and Zacharias, Joachim},
     TITLE = {Completely positive maps of order zero},
   JOURNAL = {M\"{u}nster J. Math.},
  FJOURNAL = {M\"{u}nster Journal of Mathematics},
    VOLUME = {2},
      YEAR = {2009},
     PAGES = {311--324},
}

@article {WinterZachNucDim,
    AUTHOR = {Winter, Wilhelm and Zacharias, Joachim},
     TITLE = {The nuclear dimension of {$\C$}-algebras},
   JOURNAL = {Adv. Math.},
  FJOURNAL = {Advances in Mathematics},
    VOLUME = {224},
      YEAR = {2010},
    NUMBER = {2},
     PAGES = {461--498},
}

@article{KirchbergWinterDecompRank,
    AUTHOR = {Kirchberg, Eberhard and Winter, Wilhelm},
     TITLE = {Covering dimension and quasidiagonality},
   JOURNAL = {Internat. J. Math.},
  FJOURNAL = {International Journal of Mathematics},
    VOLUME = {15},
      YEAR = {2004},
    NUMBER = {1},
     PAGES = {63--85},
}

@article {Voiculescu95,
    AUTHOR = {Voiculescu, Dan},
     TITLE = {Dynamical approximation entropies and topological entropy in
              operator algebras},
   JOURNAL = {Comm. Math. Phys.},
  FJOURNAL = {Communications in Mathematical Physics},
    VOLUME = {170},
      YEAR = {1995},
    NUMBER = {2},
     PAGES = {249--281},
}

@article {KerrLi05,
    AUTHOR = {Kerr, David and Li, Hanfeng},
     TITLE = {Dynamical entropy in {B}anach spaces},
   JOURNAL = {Invent. Math.},
  FJOURNAL = {Inventiones Mathematicae},
    VOLUME = {162},
      YEAR = {2005},
    NUMBER = {3},
     PAGES = {649--686},
}

@article {Bowen71,
    AUTHOR = {Bowen, Rufus},
     TITLE = {Entropy for group endomorphisms and homogeneous spaces},
   JOURNAL = {Trans. Amer. Math. Soc.},
  FJOURNAL = {Transactions of the American Mathematical Society},
    VOLUME = {153},
      YEAR = {1971},
     PAGES = {401--414},
}

@article{Kerr04,
    author = {Kerr, D.},
    journal = {Geom. Funct. Anal.},
    number = {3},
    pages = {575--594},
    title = {Entropy and induced dynamics on state spaces},
    volume = {14},
    year = {2004}
}

@article{Kerr07,
    author = {Kerr, David},
    journal = {Bull. Lond. Math. Soc.},
    number = {2},
    pages = {265--271},
    title = {Generically infinite entropy in a simple {AF} algebra},
    volume = {39},
    year = {2007}
}

@article{Yano80,
    author = {Yano, Koichi},
    journal = {Invent. Math.},
    number = {3},
    pages = {215--220},
    title = {A remark on the topological entropy of homeomorphisms},
    volume = {59},
    year = {1980}
}

@article{Jacelon22,
    author = {Jacelon, Bhishan},
    journal = {Forum Math. Sigma},
    pages = {Paper No. e39, 21},
    title = {Chaotic tracial dynamics},
    volume = {11},
    year = {2023}
}

@unpublished{CGSTW23,
    author = {Carri\'{o}n, Jos\'{e} R. and Gabe, Jamie and Schafhauser, Christopher and Tikuisis, Aaron and White, Stuart},
    note = {arXiv:2307.06480 [math.OA]},
    title = {Classification of $^*$-homomorphisms {{\rm I}}: {S}imple nuclear {$\C$}-algebras},
    year = {2023}
}

@book {Petersen83,
    AUTHOR = {Petersen, Karl},
     TITLE = {Ergodic theory},
    SERIES = {Cambridge Studies in Advanced Mathematics},
    VOLUME = {2},
 PUBLISHER = {Cambridge University Press, Cambridge},
      YEAR = {1983},
     PAGES = {xii+329},
}

@book {BrownOzawa08,
    AUTHOR = {Brown, Nathanial P. and Ozawa, Narutaka},
     TITLE = {{$\C$}-algebras and finite-dimensional approximations},
    SERIES = {Graduate Studies in Mathematics},
    VOLUME = {88},
 PUBLISHER = {American Mathematical Society, Providence, RI},
      YEAR = {2008},
     PAGES = {xvi+509},
}

@phdthesis{Castillejos16,
  title={Decomposable approximations and coloured isomorphisms for {$\C$}-algebras},
  author={Castillejos Lopez, Jorge},
  year={2016},
  school={University of Glasgow}
}

@article{Brown99,
  title={Topological entropy in exact {$\C$}-algebras},
  author={Brown, Nathanial P},
  journal={Math. Ann.},
  volume={314},
  number={2},
  pages={347--367},
  year={1999},
  publisher={Springer}
}

@book {HurewiczWallman48,
    AUTHOR = {Hurewicz, Witold and Wallman, Henry},
     TITLE = {Dimension {T}heory},
    SERIES = {Princeton Mathematical Series},
    VOLUME = {4},
 PUBLISHER = {Princeton University Press},
      YEAR = {1948},
     PAGES = {vii+165},
}

@article{BocaGoldstein00,
    author = {Boca, Florin P. and Goldstein, Pavle},
    journal = {Bull. London Math. Soc.},
    number = {3},
    pages = {345--352},
    title = {Topological entropy for the canonical endomorphism of {C}untz-{K}rieger algebras},
    volume = {32},
    year = {2000}
}

@article{Choda99,
    author = {Choda, Marie},
    journal = {Pacific J. Math.},
    number = {2},
    pages = {235--245},
    title = {Entropy of {C}untz's canonical endomorphism},
    volume = {190},
    year = {1999}
}

@article{SkalskiZacharias08,
    author = {Skalski, Adam and Zacharias, Joachim},
    journal = {Lett. Math. Phys.},
    number = {2-3},
    pages = {115--134},
    title = {Noncommutative topological entropy of endomorphisms of {C}untz algebras},
    volume = {86},
    year = {2008}
}

@article {CPOU21,
    AUTHOR = {Castillejos, Jorge and Evington, Samuel and Tikuisis, Aaron
              and White, Stuart and Winter, Wilhelm},
     TITLE = {Nuclear dimension of simple {$\C$}-algebras},
   JOURNAL = {Invent. Math.},
  FJOURNAL = {Inventiones Mathematicae},
    VOLUME = {224},
      YEAR = {2021},
    NUMBER = {1},
     PAGES = {245--290},
}

@article {Brown06,
    AUTHOR = {Brown, Nathanial P.},
     TITLE = {Invariant means and finite representation theory of
              {$\C$}-algebras},
   JOURNAL = {Mem. Amer. Math. Soc.},
  FJOURNAL = {Memoirs of the American Mathematical Society},
    VOLUME = {184},
      YEAR = {2006},
    NUMBER = {865},
     PAGES = {viii+105},
}

@article {TikuisisWhiteWinter17,
    AUTHOR = {Tikuisis, Aaron and White, Stuart and Winter, Wilhelm},
     TITLE = {Quasidiagonality of nuclear {$\C$}-algebras},
   JOURNAL = {Ann. of Math. (2)},
  FJOURNAL = {Annals of Mathematics. Second Series},
    VOLUME = {185},
      YEAR = {2017},
    NUMBER = {1},
     PAGES = {229--284},  
}

@article{Brown04,
    author = {Brown, Nathanial P.},
    journal = {C. R. Math. Acad. Sci. Paris},
    number = {12},
    pages = {827--829},
    title = {Characterizing type {I} {$\C$}-algebras via entropy},
    volume = {339},
    year = {2004}
}

@article {BBSTW19,
    AUTHOR = {Bosa, Joan and Brown, Nathanial P. and Sato, Yasuhiko and
              Tikuisis, Aaron and White, Stuart and Winter, Wilhelm},
     TITLE = {Covering dimension of {$\C$}-algebras and $2$-coloured
              classification},
   JOURNAL = {Mem. Amer. Math. Soc.},
  FJOURNAL = {Memoirs of the American Mathematical Society},
    VOLUME = {257},
      YEAR = {2019},
    NUMBER = {1233},
     PAGES = {vii+97},
}

@article {SatoWhiteWinter15,
    AUTHOR = {Sato, Yasuhiko and White, Stuart and Winter, Wilhelm},
     TITLE = {Nuclear dimension and {$\mathcal{Z}$}-stability},
   JOURNAL = {Invent. Math.},
  FJOURNAL = {Inventiones Mathematicae},
    VOLUME = {202},
      YEAR = {2015},
    NUMBER = {2},
     PAGES = {893--921},
}

@article {Rørdam04,
    AUTHOR = {R{\o}rdam, Mikael},
     TITLE = {The stable and the real rank of {$\mathcal{Z}$}-absorbing
              {$\C$}-algebras},
   JOURNAL = {Internat. J. Math.},
  FJOURNAL = {International Journal of Mathematics},
    VOLUME = {15},
      YEAR = {2004},
    NUMBER = {10},
     PAGES = {1065--1084},
}

@article{Winter03,
    author = {Winter, Wilhelm},
    journal = {J. Funct. Anal.},
    number = {2},
    pages = {535--556},
    title = {Covering dimension for nuclear {$\C$}-algebras},
    volume = {199},
    year = {2003}
}

@article{Lin93,
    author = {Lin, Huaxin},
    journal = {J. Funct. Anal.},
    number = {1},
    pages = {1--11},
    title = {Exponential rank of {$\C$}-algebras with real rank zero and the {B}rown--{P}edersen conjectures},
    volume = {114},
    year = {1993}
}

@article {JiangSu99,
    AUTHOR = {Jiang, Xinhui and Su, Hongbing},
     TITLE = {On a simple unital projectionless {$\C$}-algebra},
   JOURNAL = {Amer. J. Math.},
  FJOURNAL = {American Journal of Mathematics},
    VOLUME = {121},
      YEAR = {1999},
    NUMBER = {2},
     PAGES = {359--413},
}

@article{HirshbergWinterZacharias,
	author = {Hirshberg, Ilan and Winter, Wilhelm and Zacharias, Joachim},
	journal = {Comm. Math. Phys.},
	number = {2},
	pages = {637--670},
	title = {Rokhlin dimension and {$\C$}-dynamics},
	volume = {335},
	year = {2015}
}

@article{SzaboWuZacharias,
	author = {Szab\'{o}, G\'{a}bor and Wu, Jianchao and Zacharias, Joachim},
	journal = {Ergodic Theory Dynam. Systems},
	number = {8},
	pages = {2248--2304},
	title = {Rokhlin dimension for actions of residually finite groups},
	volume = {39},
	year = {2019}
}

@article {Szabo19,
    AUTHOR = {Szab\'o, G\'abor},
     TITLE = {Rokhlin dimension: absorption of model actions},
   JOURNAL = {Anal. PDE},
  FJOURNAL = {Analysis \& PDE},
    VOLUME = {12},
      YEAR = {2019},
    NUMBER = {5},
     PAGES = {1357--1396},
}

@article {BrownPedersen91,
    AUTHOR = {Brown, Lawrence G. and Pedersen, Gert K.},
     TITLE = {{$\C$}-algebras of real rank zero},
   JOURNAL = {J. Funct. Anal.},
  FJOURNAL = {Journal of Functional Analysis},
    VOLUME = {99},
      YEAR = {1991},
    NUMBER = {1},
     PAGES = {131--149},
}

@article{Ornstein70,
    author = {Ornstein, Donald},
    journal = {Advances in Math.},
    pages = {337--352},
    title = {Bernoulli shifts with the same entropy are isomorphic},
    volume = {4},
    year = {1970}
}

@article{ConnesStormer,
    author = {Connes, A. and St{\o}rmer, E.},
    journal = {Acta Math.},
    number = {3-4},
    pages = {289--306},
    title = {Entropy for automorphisms of {$\mathrm{II}_{1}$} von {N}eumann algebras},
    volume = {134},
    year = {1975}
}

@article{ConnesNarnhoferThirring,
    author = {Connes, A. and Narnhofer, H. and Thirring, W.},
    journal = {Comm. Math. Phys.},
    number = {4},
    pages = {691--719},
    title = {Dynamical entropy of {$\C$} algebras and von {N}eumann algebras},
    volume = {112},
    year = {1987}
}

@article{NeshveyevStromer00,
    author = {Neshveyev, Sergey and St{\o}rmer, Erling},
    journal = {Comm. Math. Phys.},
    number = {1},
    pages = {177--196},
    title = {The variational principle for a class of asymptotically abelian {$\C$}-algebras},
    volume = {215},
    year = {2000}
}

@article{Sinai59,
    author = {Sina\u{\i}, Yakov G.},
    journal = {Dokl. Akad. Nauk SSSR},
    pages = {768--771},
    title = {On the concept of entropy for a dynamic system},
    volume = {124},
    year = {1959}
}

@article{Kolmogorov58,
    author = {Kolmogorov, Andrey N.},
    journal = {Dokl. Akad. Nauk SSSR (N.S.)},
    pages = {861--864},
    title = {A new metric invariant of transient dynamical systems and automorphisms in {L}ebesgue spaces},
    volume = {119},
    year = {1958}
}

@article{AdlerKonheimMcAndrew,
    author = {Adler, R. L. and Konheim, A. G. and McAndrew, M. H.},
    journal = {Trans. Amer. Math. Soc.},
    pages = {309--319},
    title = {Topological entropy},
    volume = {114},
    year = {1965}
}

@article{CETW21,
	author = {Castillejos, Jorge and Evington, Samuel and Tikuisis, Aaron and White, Stuart},
	journal = {M\"{u}nster J. Math.},
	number = {2},
	pages = {265--281},
	title = {Classifying maps into uniform tracial sequence algebras},
	volume = {14},
	year = {2021}
}

@article{MatuiSato12,
	author = {Matui, Hiroki and Sato, Yasuhiko},
	journal = {Acta Math.},
	number = {1},
	pages = {179--196},
	title = {Strict comparison and {$\mathcal{Z}$}-absorption of nuclear {$\C$}-algebras},
	volume = {209},
	year = {2012}
}

@article{MatuiSato14,
	author = {Matui, Hiroki and Sato, Yasuhiko},
	journal = {Duke Math. J.},
	number = {14},
	pages = {2687--2708},
	title = {Decomposition rank of {UHF}-absorbing {$\C$}-algebras},
	volume = {163},
	year = {2014}
}

@unpublished{EGLN16,
	author = {Elliott, George A. and Gong, Guihua and Lin, Huaxin and Niu, Zhuang},
	note = {arXiv:1507.03437 [math.OA]},
	title = {On the classification of simple amenable {$\C$}-algebras with finite decomposition rank, {\rm{II}}},
	year = {2016}
}

@article{CastillejosEvington20,
	author = {Castillejos, Jorge and Evington, Samuel},
	journal = {Anal. PDE},
	number = {7},
	pages = {2205--2240},
	title = {Nuclear dimension of simple stably projectionless {$\C$}-algebras},
	volume = {13},
	year = {2020}
}

@article{Winter12,
	author = {Winter, Wilhelm},
	journal = {Invent. Math.},
	number = {2},
	pages = {259--342},
	title = {Nuclear dimension and {$\mathcal{Z}$}-stability of pure {$\C$}-algebras},
	volume = {187},
	year = {2012}
}

@article{Tikuisis12,
	author = {Tikuisis, Aaron},
	journal = {Math. Ann.},
	number = {3-4},
	pages = {729--778},
	title = {Nuclear dimension, {$\mathcal{Z}$}-stability, and algebraic simplicity for stably projectionless {$\C$}-algebras},
	volume = {358},
	year = {2014}
}

@article{Schafhauser20,
	author = {Schafhauser, Christopher},
	journal = {J. Reine Angew. Math.},
	pages = {291--304},
	title = {A new proof of the {T}ikuisis-{W}hite-{W}inter theorem},
	volume = {759},
	year = {2020}
}

@article{GlasnerWeiss01,
    author = {Glasner, Eli and Weiss, Benjamin},
    journal = {Amer. J. Math.},
    number = {6},
    pages = {1055--1070},
    title = {The topological {R}ohlin property and topological entropy},
    volume = {123},
    year = {2001}
}

@book{Walters82,
	author = {Walters, Peter},
	publisher = {Springer-Verlag, New York-Berlin},
	series = {Graduate Texts in Mathematics},
	title = {An introduction to ergodic theory},
	volume = {79},
	year = {1982}
}
\end{document}